\documentclass[a4paper,reqno,12pt]{article}

\usepackage{amsmath,amssymb,amsthm,mathtools,bm}
\usepackage{geometry}
\usepackage{microtype}
\usepackage{enumitem}
\usepackage{booktabs}
\usepackage{array}
\usepackage{aliascnt}
\usepackage{fancyhdr}
\usepackage[numbers,sort&compress]{natbib}

\usepackage[hidelinks]{hyperref}
\usepackage[nameinlink,noabbrev]{cleveref}
\allowdisplaybreaks[3]
\numberwithin{equation}{section}
\setlist{itemsep=0.2em,topsep=0.35em}

\makeatletter
\renewcommand\section{\@startsection{section}{1}{\z@}%
  {-2.2ex \@plus -0.5ex \@minus -0.2ex}%
  {1.1ex \@plus 0.2ex}%
  {\normalfont\large\bfseries\raggedright}}
\makeatother

\newtheorem{theorem}{Theorem}[section]
\newaliascnt{lemma}{theorem}
\newtheorem{lemma}[lemma]{Lemma}
\aliascntresetthe{lemma}
\newaliascnt{proposition}{theorem}
\newtheorem{proposition}[proposition]{Proposition}
\aliascntresetthe{proposition}
\newaliascnt{corollary}{theorem}

\aliascntresetthe{corollary}
\theoremstyle{remark}
\newaliascnt{remark}{theorem}
\newtheorem{remark}[remark]{Remark}
\aliascntresetthe{remark}
\crefname{theorem}{theorem}{theorems}
\Crefname{theorem}{Theorem}{Theorems}
\crefname{lemma}{lemma}{lemmas}
\Crefname{lemma}{Lemma}{Lemmas}
\crefname{proposition}{proposition}{propositions}
\Crefname{proposition}{Proposition}{Propositions}
\crefname{corollary}{corollary}{corollaries}
\Crefname{corollary}{Corollary}{Corollaries}
\crefname{remark}{remark}{remarks}
\Crefname{remark}{Remark}{Remarks}
\crefname{equation}{equation}{equations}
\Crefname{equation}{Equation}{Equations}
\crefname{section}{section}{sections}
\Crefname{section}{Section}{Sections}

\newcommand{\dd}{\,\mathrm d}
\newcommand{\norm}[1]{\left\lVert #1\right\rVert}

\hypersetup{
  pdftitle={ Large-Time Behavior for One-Dimensional Planar Compressible Magnetohydrodynamic Equations on Unbounded Domains},
  pdfauthor={Jing Li; Ziyin Liu},
  pdfkeywords={Planar compressible magnetohydrodynamic equations, unbounded domains, Cauchy problem, Neumann boundary conditions, Dirichlet boundary conditions, uniform-in-time estimates, large-time behavior}
}

\title{\bfseries  Large-Time Behavior
for One-Dimensional Planar Compressible Magnetohydrodynamic Equations
on Unbounded Domains\thanks{This work is partially supported by the National Natural Science Foundation of
China (12471227, 12494541), and the Double-Thousand Plan of Jiangxi Province (jxsq2023201115).}}
\author{Jing Li$^a$; Ziyin Liu$^b$
\thanks{a. Institute of Applied Mathematics \& State Key Laboratory of Mathematical Sciences, AMSS,\\ \& Hua Loo-Keng Key Laboratory of Mathematics,
Chinese Academy of Sciences,\\ Beijing, 100190 People's Republic of China, \\ and School of Mathematics and Computer Sciences, \\ Institute of Mathematics and Interdisciplinary Sciences,\\ Nanchang
University, Nanchang, 330031, People's Republic of China (ajingli@gmail.com).\\
b. Institute of Applied Mathematics, Academy of Mathematics and Systems Science,\\
Chinese Academy of Sciences, Beijing 100190, People's Republic of China\\
(a1909462594@gmail.com).}}
\date{}

\begin{document}
\maketitle

\begin{abstract} 

Global existence and the large-time behavior of strong solutions to the
one-dimensional planar compressible magnetohydrodynamic system in unbounded
domains are investigated.
The longitudinal  viscosity is assumed to be a positive constant plus  any nonnegative power  of the density, or to depend on temperature through a power law. The heat conductivity is taken to be proportional to any nonnegative power of temperature, while the transverse viscosity and magnetic diffusivity
are taken to be positive constants. Global existence, uniform-in-time estimates, and convergence to equilibrium are established  without imposing any smallness conditions on the initial data. In the density-dependent case, any fixed
nonnegative viscosity exponent is allowed; in the temperature-dependent
case, the nonnegative viscosity exponent is required to be sufficiently
small in terms of the initial data. An auxiliary elliptic equation is
introduced to control magnetic-pressure oscillations through a
square-root-in-time estimate. Combined with the estimates appropriate to
each viscosity law, this yields uniform upper and positive lower bounds
for the specific volume and temperature and asymptotic stability of the
global strong solution.
\end{abstract}



\section{Introduction}\label{sec:intro}

In Lagrangian coordinates, the one-dimensional equations of planar
compressible magnetohydrodynamics (MHD) take the form \cite{Cabannes:1970,JeffreyTaniuti1964,LandauLifshitz1984}
\begin{align}
 v_t&=u_x, \label{eq:mass}\\
 u_t+\left(P+\frac12|b|^2\right)_x
 &=\left(\frac{\mu u_x}{v}\right)_x, \label{eq:mom}\\
 w_t-b_x&=\left(\frac{\lambda w_x}{v}\right)_x,
                                                        \label{eq:trans}\\
 (vb)_t-w_x&=\left(\frac{\nu b_x}{v}\right)_x, \label{eq:ind} 
            \end{align} 
            \begin{align}                                        
 &\left(e+\frac{u^2+|w|^2+v|b|^2}{2}\right)_t
+\left\{u\left(P+\frac12|b|^2\right)
 -w\cdot b\right\}_x\notag\\
 &=\left\{\frac{\kappa(\theta)\theta_x}{v}
 +\frac{\mu uu_x}{v}
 +\frac{\lambda w\cdot w_x}{v}
 +\frac{\nu b\cdot b_x}{v}\right\}_x,                  \label{eq:energy}
\end{align}
for $x\in\Omega$ and $t>0$, where $\Omega$ is either the whole line
$\mathbb R$ or the half-line $\mathbb R_+=(0,\infty)$.
Here $v$, $u$, $\theta$, $e$, and $P$ denote the specific volume,
longitudinal velocity, absolute temperature, specific internal energy,
and pressure, respectively, while $w,b\in\mathbb R^2$ denote the
transverse velocity and magnetic field, respectively.
We consider an ideal gas with
\begin{equation}
 P=\frac{R\theta}{v},\qquad
 e=c_v\theta+\mathrm{const},                            \label{eq:gas}
\end{equation}
where $R>0$ is the gas constant and $c_v>0$ is the specific heat
at constant volume. The transverse shear viscosity coefficient
$\lambda>0$ and the magnetic diffusivity $\nu>0$ are constants.
The heat conductivity is given by
\begin{equation}
 \kappa(\theta)=\widetilde\kappa\theta^\beta,\qquad
 \widetilde\kappa>0,\quad \beta\geq0.                  \label{eq:coeff}
\end{equation}
The longitudinal viscosity is given by either of the following two alternative laws:
\begin{equation}\label{eq:viscosity}
\begin{aligned}
\eqref{eq:viscosity}_1\quad&\mu=\mu(v)=\mu_1+\mu_2v^{-\alpha},
\qquad \alpha\geq0,\\
\eqref{eq:viscosity}_2\quad&\mu=\mu(\theta)=\mu_0\theta^\gamma,
\qquad \gamma\geq0.
\end{aligned}
\end{equation}
where $\mu_0,\mu_1,\mu_2$ are positive constants.
In $\eqref{eq:viscosity}_1$, the exponent $\alpha$ is arbitrary but fixed,
and no smallness condition is imposed on $\frac{\mu_2}{\mu_1}$.
In $\eqref{eq:viscosity}_2$, $\gamma$ is required to be sufficiently small,
with a threshold depending on the initial data. In both cases,
$\beta\geq0$ is arbitrary but fixed.

We prescribe the initial data
\begin{equation}
 (v,u,\theta,b,w)(x,0)
 =(v_0,u_0,\theta_0,b_0,w_0)(x),\qquad x\in\Omega.
                                                        \label{eq:data}
\end{equation}
We consider the following three problems.
For the Cauchy problem, $\Omega=\mathbb R$, and the far-field condition is
\begin{equation}
 (v,u,\theta,b,w)(x,t)\longrightarrow(1,0,1,0,0)
 \quad (x\to\pm\infty).                               \label{eq:bc-C}
\end{equation}
The two half-line problems are posed on $\Omega=\mathbb R_+$.
For homogeneous Neumann conditions on the magnetic field, we impose
\begin{equation}
 \begin{gathered}
 u(0,t)=w(0,t)=0,\qquad \theta_x(0,t)=b_x(0,t)=0,\\ 
 (v,u,\theta,b,w)(x,t)\longrightarrow(1,0,1,0,0)
 \quad (x\to+\infty).
 \end{gathered}                                         \label{eq:bc-N}
\end{equation}
For homogeneous Dirichlet conditions on the magnetic field, we impose
\begin{equation}
 \begin{gathered}
 u(0,t)=w(0,t)=b(0,t)=0,\qquad \theta_x(0,t)=0,\\
 (v,u,\theta,b,w)(x,t)\longrightarrow(1,0,1,0,0)
 \quad (x\to+\infty).
 \end{gathered}                                         \label{eq:bc-D}
\end{equation}
Thus, the boundary at $x=0$ is thermally insulated in both half-line
problems. We assume that the initial data satisfy
\begin{equation}
 (v_0-1,u_0,\theta_0-1,b_0,w_0)\in H^1(\Omega),
 \qquad
 \inf_\Omega v_0>0,\quad \inf_\Omega\theta_0>0,       \label{eq:init}
\end{equation}
and, in the half-line cases, satisfy $u_0(0)=w_0(0)=0$ and,
for \eqref{eq:bc-D}, $b_0(0)=0$. The Neumann conditions are
imposed on the solution for positive times; no initial derivative
trace is required at the $H^1$ level. The initial perturbation may be
arbitrarily large in $H^1(\Omega)$.

Kazhikhov \cite{Kazhikhov1987} obtained a priori estimates for a
one-dimensional MHD model with constant transport coefficients and
large initial data in a bounded interval. Chen-Wang \cite{ChenWang2002} studied a free-boundary
problem with general constitutive laws and transport coefficients,
while Wang \cite{Wang2003} considered the initial-boundary value problem
on a fixed bounded interval. For temperature-dependent heat
conductivity, Hu-Ju \cite{HuJu2015} and Huang-Shi-Sun
\cite{HuangShiSun2019} proved global existence of strong solutions for
$\beta>0$. For constant longitudinal and transverse viscosities and
constant heat conductivity, Li-Shang
\cite{LiShang2019} proved global existence for a class of
magnetic diffusivities depending on the specific volume.
For constant viscosities and magnetic diffusivity, Huang-Shi-Sun
\cite{HuangShiSun2021} proved uniform-in-time positive lower and upper
bounds for the specific volume and temperature, as well as exponential
convergence to equilibrium, for every $\beta\geq0$. Their result
concerns a bounded interval with homogeneous Dirichlet conditions for
the velocities and magnetic field and thermally insulated boundaries. 
On bounded intervals, Song--Zhao \cite{SongZhao2025} considered
viscosities and magnetic diffusivity proportional to $\theta^\alpha$,
with heat conductivity proportional to $\theta^\beta$, $\beta>0$.
For large $H^2$ initial data and sufficiently small $\alpha$, they
proved global existence and exponential convergence to equilibrium.
For $\eqref{eq:viscosity}_2$, only the longitudinal viscosity depends on
the temperature, while the magnetic diffusivity is an arbitrary
fixed positive constant.

For the density-dependent viscosity $\eqref{eq:viscosity}_1$,
Cao-Peng-Sun \cite{CaoPengSun2021} proved
global existence of strong solutions to the Cauchy problem and the
half-line problem with Dirichlet conditions for the magnetic field
for $\alpha\geq0$ and $\beta>0$. L\"u-Shi-Xiong \cite{LuShiXiong2021} treated
these problems with constant viscosities and heat conductivity.
For the half-line problem with Neumann conditions on both the
magnetic field and the temperature, Tong-Wang-Zhang
\cite{TongWangZhang2026} proved global existence of strong solutions
for $\alpha=0$ and $\beta\geq0$.
In each of these MHD existence results, the specific volume and
temperature are bounded above and away from zero on every finite
time interval, with constants that may depend on the final time.

 Concerning the large-time behavior of strong solutions on unbounded domains, 
for the compressible Navier--Stokes system with constant viscosity and heat
conductivity on unbounded domains, Jiang \cite{Jiang1999,Jiang2002}  obtained positive lower and upper bounds
for the specific volume uniform in both space and time.
Li-Liang \cite{LiLiang2016} subsequently proved
uniform-in-time estimates and convergence to equilibrium for large
initial data. Li-Xu
\cite{LiXu2026} obtained similar results when the viscosity is constant
and the heat conductivity is given by
$\kappa(\theta)=\widetilde\kappa\theta^\beta$ with $\beta>0$. Results on the large-time behavior for the MHD Cauchy problem are also available under alternative assumptions on the transport coefficients.
Shang \cite{Shang2022} considered
viscosities and heat conductivity with a common factor $h(v)\theta^a$,
subject to growth conditions on $h$.
For sufficiently small $|a|$ and sufficiently large magnetic
diffusivity, Shang established global existence and large-time
stability of smooth solutions.
More recently, Li-Liao-Wan \cite{LiLiaoWan2026}
studied the case of constant longitudinal viscosity and heat
conductivity of the form
\[
 \kappa(v,\theta)=\kappa_1+\kappa_2v\theta^{\mathfrak b},
 \qquad \kappa_1,\kappa_2>0,\quad \mathfrak b>\frac32.
\]
They proved that the specific volume and temperature remain bounded
above and away from zero uniformly in time and that the solutions
converge to equilibrium. The authors also noted that these results
extend to the half-line problem with Dirichlet conditions for
the magnetic field, including the case of thermal insulation at
the boundary.

We begin with the local existence result needed for all the coefficient
and boundary conditions considered here. A proof at the $H^1$ level
is given in Appendix~\ref{sec:local}.
\begin{lemma}\label{lem:01a}
Let $\beta\geq0$ be fixed.
For $\eqref{eq:viscosity}_1$, let $\alpha\geq0$ be fixed; for
$\eqref{eq:viscosity}_2$, let $0\leq\gamma\leq1$. Assume \eqref{eq:init} and, on the half-line,
the initial Dirichlet conditions stated above.
For each of the three conditions \eqref{eq:bc-C}--\eqref{eq:bc-D},
the corresponding problem \eqref{eq:mass}--\eqref{eq:data},
with exactly one of the two viscosity laws in \eqref{eq:viscosity},
has a unique local strong solution
on $[0,T_0]$ for some $T_0>0$, with
\begin{equation}\label{eq:finite}
\begin{cases}
 v-1,u,\theta-1,b,w\in C([0,T_0];H^1(\Omega)),\\
 v_t\in C([0,T_0];L^2(\Omega))\cap L^2(0,T_0;H^1(\Omega)),\\
 u_t,w_t,b_t,\theta_t,u_{xx},w_{xx},b_{xx},\theta_{xx}
 \in L^2(\Omega\times(0,T_0)).
\end{cases}
\end{equation}
Moreover, $v$ and $\theta$ remain bounded above and away from zero.
The solution can be continued as long as
$\|(v-1,u,\theta-1,w,b)\|_{H^1(\Omega)}$ stays bounded and
$v$ and $\theta$ stay bounded away from zero.
The local existence time depends only on these initial bounds,
the fixed coefficients, $\beta$, and the problem under consideration.
In case $\eqref{eq:viscosity}_1$, it may also depend on the fixed $\alpha$;
in case $\eqref{eq:viscosity}_2$, it can be chosen uniformly for
$\gamma\in[0,1]$. The Neumann conditions hold for almost every
positive time in the trace sense; no initial derivative trace is required.
\end{lemma}

For these problems, the main question is whether the bounds can be
made independent of time and whether large solutions converge to
equilibrium. Classical small-data results are available; see
\cite{KawashimaOkada1982,LiuZeng1997,QinLiuWang2015} and the
references therein. The finite-time estimates in
\cite{CaoPengSun2021,LuShiXiong2021,TongWangZhang2026} do not
yield the large-data convergence considered here.

In this paper, we establish global existence, uniform-in-time
estimates, and convergence to equilibrium for all three problems,
without any smallness condition on the initial perturbation.
The result holds for every fixed $\beta\geq0$, with arbitrary fixed
$\alpha\geq0$ in $\eqref{eq:viscosity}_1$ and sufficiently small
$\gamma\geq0$ in $\eqref{eq:viscosity}_2$. It includes constant viscosity
and constant heat conductivity.

Our main result is as follows.

\begin{theorem}\label{thm:main} Under the assumptions of Lemma~\ref{lem:01a}, the following assertions hold.
For $\eqref{eq:viscosity}_1$, let $\alpha\geq0$ be arbitrary but fixed.
For $\eqref{eq:viscosity}_2$, there exists $\gamma_0\in(0,1]$, depending only
on the fixed coefficients, $\beta$, the problem under consideration,
the initial $H^1$ norm and the initial positive lower bounds,
such that the assertions hold whenever $0\leq\gamma\leq\gamma_0$.
In either case, there exists a unique
global strong solution satisfying \eqref{eq:finite} with $T_0$
replaced by any $T>0$. Moreover, it satisfies
\begin{equation}
 0<c\leq v(x,t),\theta(x,t)\leq C,
 \qquad x\in\Omega,\quad t\geq0,                        \label{eq:mainvt}
\end{equation}
\begin{equation}
 \sup_{t\geq0} 
 \norm{ (v -1,u ,\theta -1,w ,b)( \cdot,t) }_{H^1(\Omega)}^2 
 \leq C,                                      \label{eq:mainH1}
\end{equation}
and
\begin{align}
 \int_0^\infty \bigl\{ \norm{ v_x (\cdot,t)}_{L^2(\Omega)}^2 + \norm{ ( u_x ,\theta_x, w_x ,b_x)( \cdot,t) }_{H^1(\Omega)}^2+\norm{ ( u_t ,\theta_t, w_t ,b_t)( \cdot,t) }_{L^2(\Omega)}^2 \bigr\}
 \dd t\leq C.                                  \label{eq:mainds}
\end{align}
Moreover, for every $2<p\leq\infty$,
\begin{equation}
\norm{ (v -1,u ,\theta -1,w ,b)( \cdot,t) }_{L^p(\Omega)}+ \norm{(v_x,u_x,\theta_x,w_x,b_x)(\cdot,t)}_{L^2(\Omega)} 
 \longrightarrow0\quad (t\to\infty).                 \label{eq:maingr}
\end{equation}
Here $c$ and $C$ may depend on the fixed coefficients, $\beta$,
the problem under consideration, and the initial data, but are
independent of time. In case $\eqref{eq:viscosity}_1$, they may also depend
on the fixed $\alpha$; in case $\eqref{eq:viscosity}_2$, they are uniform
for $\gamma\in[0,\gamma_0]$. The convergence is asserted for each
fixed admissible exponent $\alpha$ or $\gamma$, respectively.
\end{theorem}

A few remarks are in order.
\begin{remark}\label{r:large-data}   It should be emphasized that no smallness condition on the initial data is required in our main Theorem \ref{thm:main}, while smallness assumptions are imposed in the previous related works (see 
\cite{KawashimaOkada1982,LiuZeng1997,QinLiuWang2015} and references therein). More precisely, for every fixed $\beta\geq0$ and for $\alpha$ or $\gamma$
in the range specified for the corresponding viscosity law, we establish uniform-in-time estimates and the asymptotic convergence of solutions toward equilibrium  for arbitrarily large initial data. In particular, our Theorem \ref{thm:main} covers the classical case where all transport coefficients are positive constants. To our knowledge, the large-data asymptotic stability for this classical system has not been established previously for the Cauchy and half-line problems considered in this paper.
\end{remark}

\begin{remark}\label{r:ns-extension} For the Cauchy problem and the thermally insulated half-line problem, our results extend the large-data stability theory of compressible Navier--Stokes equations to the planar compressible MHD system. This generalizes the previous results established by Li--Liang \cite{LiLiang2016} for the constant heat conductivity case \((\beta=0)\) and by Li--Xu \cite{LiXu2026} for temperature-dependent heat conductivity with \(\beta>0\), where \(\kappa(\theta)=\widetilde\kappa\theta^\beta\). In particular, the present extension allows for arbitrarily large initial perturbations in the transverse velocity and magnetic field.
\end{remark}
\begin{remark}\label{r:ns-temp-viscosity}
For the Cauchy problem of the one-dimensional compressible Navier--Stokes equations with temperature-dependent transport coefficients, Dong--Guo \cite{DongGuo2025} established global existence and asymptotic stability for large initial data under $\mu(\theta)=\theta^a$ and $\kappa(\theta)=\theta^b$, where $b\geq0$ is arbitrary and $a\geq0$ is sufficiently small. Their result requires $v_0-1\in H^2(\mathbb R)$ and $(u_0,\theta_0-1)\in H^1(\mathbb R)\times H^1(\mathbb R)$. In the notation of the present paper, $a$ corresponds to $\gamma$ and $b$ to $\beta$. Our Theorem~\ref{thm:main} improves the regularity requirement on the initial specific volume from $H^2$ to $H^1$ and extends this large-data stability theory to the planar compressible MHD system. In addition, the present result covers the corresponding thermally insulated half-line problems and allows arbitrarily large initial perturbations in the transverse velocity and magnetic field.
\end{remark}

We now outline the main difficulties and key estimates involved in the proof. Principal challenges arise from two key aspects in the present large-data unbounded-domain problem. First, the unboundedness of the domain creates essential difficulties, which have been discussed in \cite{Shang2022,TongWangZhang2026,XiaoLu2021,ZhongXie2023}. As pointed out in \cite{Shang2022}, the Poincar\'{e} inequality valid for bounded intervals is no longer available, so the conventional magnetic-field estimates cannot be carried over directly. Second, in the large-data setting, the strong magnetic coupling effect must be rigorously controlled without any smallness assumption on the initial perturbation. The two viscosity laws lead to different estimates for $v$.
For $\eqref{eq:viscosity}_1$, the density-dependent part of the viscosity
also contributes an additional term to the elliptic estimate.
For $\eqref{eq:viscosity}_2$, derivatives of the viscosity produce terms
that must be controlled by the smallness of $\gamma$.

In case $\eqref{eq:viscosity}_1$, we introduce $M(v)$ in \eqref{eq:Mv},
so that $M(v)_t=\frac{\mu u_x}{v}$. Normalizing $Z_N$ in
\eqref{eq:Zdef} by its cell integral $\overline Z_N$ cancels the
cell average of the stress $\sigma$ in \eqref{eq:sigma} and gives
\eqref{eq:Z-evol}. The weighted mean in this identity is bounded
by the energy estimate. For a sufficiently small fixed $\delta>0$, define
\[
 \mathcal B\triangleq\left\{t\geq0:
 \int_\Omega\frac{\theta^{\beta-2}\theta_x^2}{v}\dd x>\delta
 \right\}.
\]
When $v$ is sufficiently small and $t\notin\mathcal B$,
\eqref{eq:good-two} gives $\theta\geq c_\theta>0$.
Together with the uniform upper bound for the weighted mean
in \eqref{eq:Zmean}, this makes the right-hand side of
\eqref{eq:Z-evol} nonnegative. For $t\in\mathcal B$, the same
right-hand side is still bounded below by $-C$, since
$\theta/v$ and $|b|^2/2$ are nonnegative.The measure bound in
\eqref{eq:bad-measure} therefore yields the uniform positive lower
bound for $v$, without using a uniform lower bound for $\theta$;
see Proposition~\ref{p:vlow}.

The next point is to obtain a uniform positive lower bound for the
temperature. For a sufficiently large fixed constant $A_0$ and $p>2$,
set
\[
 Y_p(t)=\norm{(\theta^{-1}(t)-A_0)_+}_{L^{p+1}(\Omega)}.
\]
Using the uniform positive lower bound \eqref{eq:vlow} for $v$, the truncation argument
in \cite{CaoPengSun2021,LuShiXiong2021,TongWangZhang2026} gives
$Y_p'(t)\leq C$ for almost every $t$ with $Y_p(t)>0$
(see \eqref{eq:Ypdt}), where $C$ is independent of
$p$ and time. Integrating this estimate over the whole time interval
would introduce a term $Ct$. To remove this time dependence, we
observe that the local mean and oscillation estimates in
Lemmas~\ref{l:cells} and~\ref{l:osc} yield a fixed positive lower bound
for $\theta$ for almost every $t\notin\mathcal B$; see \eqref{eq:good}.
Thus,
the choice of $A_0$ in \eqref{eq:invcut} ensures that $Y_p(0)=0$
and $Y_p(t)=0$ for almost every
$t\notin\mathcal B  $. Moreover, the energy
estimate \eqref{eq:ent-id} implies that the measure of $\mathcal B$ is bounded. Thus, if $Y_p(t)>0$, let $(s,r)$ be the connected
component of $\{Y_p>0\}$ containing $t$. Since $Y_p(s)=0$, integrating
\eqref{eq:Ypdt} only from $s$ to $t$ gives
\[
 Y_p(t)\leq C(t-s)\leq C|\mathcal B  |\leq C,
\]
with a constant independent of $p$ and $t$. Letting $p\to\infty$
yields the uniform positive lower bound \eqref{eq:t-low} for $\theta$;
see Proposition~\ref{p:tlow} for details.

For $\eqref{eq:viscosity}_1$, a further difficulty lies in establishing a uniform-in-time upper
bound for the specific volume. At a point where $Z_N$ attains its
maximum on $[N,N+1]$, the difference between $\frac{\theta}{v}$ and its
average with weight $Z_N$ in \eqref{eq:Z-evol} is
bounded above by a negative constant whenever $\frac{Z_N}{\overline Z_N}$
is sufficiently large and $t\notin\mathcal B$. The time integral
of this difference over $\mathcal B$ has a uniform upper bound,
by \eqref{eq:max-pressure} and \eqref{eq:bad-measure}.
The magnetic term is the difference
between $\frac12|b|^2$ and its average with weight $Z_N$ over
$[N,N+1]$, and is bounded above by
$\int_N^{N+1}|b||b_y|\dd y$.
Integrating the resulting differential inequality gives
\eqref{jhoa1}, with a negative term proportional to $t-s$.
It thus remains to estimate
$\int_s^t\!\int_N^{N+1}|b||b_y|\dd y\dd r$. The
Cauchy--Schwarz inequality bounds this integral by
\begin{equation*}
 C\left(\int_s^t\!\int_\Omega v\theta|b|^2\dd x\dd r\right)^{\frac12}
 \left(\int_s^t\!\int_\Omega\frac{|b_x|^2}{v\theta}\dd x\dd r\right)^{\frac12}.
\end{equation*}
The second factor is bounded by \eqref{eq:ent-id}. It therefore
suffices to prove
\begin{equation*}
 \int_s^t\!\int_\Omega v\theta|b|^2\dd x\dd r\leq C(1+t-s).
\end{equation*}

To this end, we use the normalization in Section~\ref{sec:lower}
and set $\Phi(z)=z-\log z-1$. The relative energy density
\begin{equation}
 \eta\triangleq\frac{u^2+|w|^2+v|b|^2}{2}+\Phi(v)+\Phi(\theta)
      \label{eq:eta} 
\end{equation} satisfies (see \eqref{eq:ent})
\begin{equation}
 \eta_t-\left(\frac{\left(\mathcal K_\beta(\theta)\right)_x}{v}\right)_x =\cdots,
                                                               \label{eq:ent-I}
\end{equation} where 
\begin{equation}
 \mathcal K_\beta(\xi)
 \triangleq\int_1^\xi s^{\beta-1}(s-1)\dd s,
 \qquad \xi>0.                                  \label{eq:K-beta}
\end{equation}
By Remark~\ref{remark:k},
\begin{equation}
 v\theta|b|^2
 \leq C\eta+C\mathcal K_\beta(\theta)\eta.       \label{eq:vtb}
\end{equation}
Thus, it suffices to estimate
$\int_\Omega\mathcal K_\beta(\theta)\eta\dd x$.
For this purpose, we introduce an  auxiliary elliptic equation
\begin{equation}\label{eq:aux1}
 -\left(\frac{\psi_x}{v}\right)_x+\psi=\eta, \quad \psi\in H^1(\Omega),
\end{equation} which plays an important role in our analysis. Multiplying
\eqref{eq:ent} by $\psi$ gives \eqref{eq:res-id}. In this
identity, the constant part of the viscosity yields the positive term
$\frac{\mu_1}{2}\int_\Omega\eta u^2\dd x$, while its density-dependent
part leaves the additional integral
\begin{equation*}
 I_{\alpha}=\mu_2\int_\Omega v^{-\alpha}Fuu_x\dd x,
 \qquad F=-\frac{\psi_x}{v}.
\end{equation*}
We estimate $I_{\alpha}$ directly by \eqref{eq:Ialpha-est}, without
differentiating $\mu$. The truncated temperature estimate
\eqref{eq:heat-conv} then allows this term to be absorbed after
integration in time. Consequently,
\begin{equation*}
 \int_s^t\!\int_\Omega\mathcal K_\beta(\theta)\eta\dd x\dd r
 \leq C(1+t-s).
\end{equation*}
This yields the required bound for $v\theta|b|^2$ and hence the
square-root estimate \eqref{eq:b-osc}. Its growth is dominated by
the negative term proportional to $t-s$ in \eqref{jhoa1}, which
gives the upper bound for $v$; see Lemmas~\ref{l:res} and
\ref{l:bP} and Propositions~\ref{p:bcap} and~\ref{p:v-up}.
The argument places no upper restriction on $\alpha$.

For $\eqref{eq:viscosity}_2$, division of the momentum equation by $\mu$
gives \eqref{eq:sigma-T}, with the remainder \eqref{eq:hmu-T}.
Under \eqref{eq:boot} and \eqref{eq:small-gamma}, the viscosity is
bounded above and away from zero, and the remainder satisfies
\eqref{eq:hmu-bd-T}. The representation formula
\eqref{jhoa1-T} therefore contains an additional term
$C\varepsilon_{\gamma}(1+t-s)$ in the exponent.
In the elliptic estimate, differentiating $\mu$ produces
$-\frac12\int_\Omega\mu_xFu^2\dd x$; its time integral is
controlled by \eqref{eq:mu-u2-T}. Together with
\eqref{eq:bP-mu-T}, this gives the upper bound for $v$ after
choosing $\varepsilon_{\gamma}$ sufficiently small.
The same representation formula then gives the positive lower
bound for $v$, and Proposition~\ref{p:tlow} applies to $\theta$.
The estimates in this case are derived under the temporary bounds
\eqref{eq:boot}; these assumptions are removed by the choice
\eqref{eq:gamma-choice} and a continuation argument.

Once the upper bound for $v$ is available, a natural next step is to
estimate $M(v)_x$ in case $\eqref{eq:viscosity}_1$ and $(\log v)_x$
in case $\eqref{eq:viscosity}_2$. The corresponding energy estimates require
control of
\begin{equation*}
 \int_0^\infty\!\int_\Omega
 \left\{u_x^2+\frac{\theta_x^2}{ \theta}
 +|b|^2|b_x|^2\right\}\dd x\dd t.
\end{equation*} 
To this end, we   first construct a bounded, nondecreasing function $g_\beta$
that agrees with $2(1-\frac1\theta)$ at low and moderate temperatures
and is suitably modified at high temperatures; see \eqref{eq:capdef}. The choice of  $g_\beta$  allows us to estimate $  w_x  $ and $ b_x $
before estimating $v_x$ or obtaining an upper bound for $\theta$;
see Proposition \ref{p:cap}. In particular, \eqref{eq:tr-max} and \eqref{eq:b-inf} give
\begin{align*}
 &\int_0^\infty\!\int_\Omega|b|^2|b_x|^2\dd x\dd t\leq C,\\
 &\int_0^\infty\left(
 \norm{b_x(t)}_{L^\infty(\Omega)}^2
 +\norm{b_x(t)}_{L^\infty(\Omega)}^4
 +\norm{b(t)}_{L^\infty(\Omega)}^4\right)\dd t\leq C.
\end{align*}
When $\beta>0$, Proposition \ref{p:cap} also yields an estimate for the space-time integral  of $\frac{\theta_x^2}{ \theta}$.
For $\beta=0$, we   adapt  the truncation
argument of Li-Liang \cite{LiLiang2016} on the region where the
temperature is large to obtain an estimate for the space-time integral  of $ \theta_x^2 $.   The additional heating terms involving $w_x$
and $b_x$ are controlled by \eqref{eq:tr-max}, while the magnetic
pressure terms are handled by Young's inequality and \eqref{eq:b-inf};
see Proposition~\ref{p:hot0}.
We can then obtain the estimate for $v_x$, using \eqref{eq:lj09}
for $\eqref{eq:viscosity}_1$ and \eqref{eq:lj09-T} for
$\eqref{eq:viscosity}_2$; see Propositions \ref{p:cap} and \ref{p:hot0} and  Lemmas \ref{l:flux} and  \ref{l:vx}.

The upper bounds for the  temperature  in
\cite{CaoPengSun2021,LuShiXiong2021,TongWangZhang2026} depend on
$T$. Here, the key input is \eqref{eq:tr-max}, which holds on
$(0,\infty)$ and is obtained before the upper bound for $\theta$.
This allows us to adapt the temperature estimates
of Li--Xu \cite{LiXu2026} to control the additional magnetic
heating terms with constants independent of time.
After some careful analysis, using \eqref{eq:mechwt},
\eqref{eq:Y-L1}, and Gronwall's inequality, we obtain the uniform
upper bound \eqref{eq:t-up}; see Proposition~\ref{p:thig}
for details.

The rest of the paper is organized as follows. Section~2 contains
the basic estimates. In Section~3, we derive the uniform bounds
for $v$ and the positive lower bound for $\theta$, treating the
two viscosity laws separately after introducing the common
elliptic estimates. Section~4 is devoted to the remaining
uniform estimates. In Section~5, we establish global existence
and large-time behavior and complete the proof of the main theorem.

\section{Basic estimates}\label{sec:lower}

Let $T_*$ be the maximal existence time supplied by
Lemma~\ref{lem:01a}, and fix $T<T_*$. All estimates in
Sections~\ref{sec:lower}--\ref{sec:high}, and in
Proposition~\ref{p:full}, are a priori estimates on $[0,T]$.
To simplify notation, we give the proof for
\[
 R=c_v=\lambda=\nu=\widetilde\kappa=1,
\]
and additionally take $\mu_0=1$ in case $\eqref{eq:viscosity}_2$.
The same estimates hold for general fixed positive coefficients,
with the corresponding constant weights retained.
In case $\eqref{eq:viscosity}_1$, we consider $\alpha>0$, allowing
$\mu_2=0$ in the estimates. The case $\alpha=0$ is included by
replacing $\mu_1$ with $\mu_1+\mu_2$ and setting $\mu_2=0$.
An integer $N$ is admissible if $N\in\mathbb Z$ on the whole line
and $N\in\mathbb N_0$ on the half-line. Set
\[
 \Phi(z)=z-\log z-1,\qquad z>0.
\]
Constants may depend on the fixed positive coefficients, $\beta$,
the problem under consideration, and the initial data, including
the positive lower bounds in \eqref{eq:init}. In case
$\eqref{eq:viscosity}_1$, they may also depend on the fixed $\alpha$.
In case $\eqref{eq:viscosity}_2$, they are independent of $\gamma$
in the range considered below and of the temporary constant $M$.
Unless stated otherwise, they are independent of $T$ and $N$.
Spatial norms with the domain omitted are taken over $\Omega$.
The calculations are justified in Appendix~\ref{sec:local}.

In case $\eqref{eq:viscosity}_2$, we assume temporarily that the local solution on $[0,T]$
satisfies, for some $M\geq2$,
\begin{equation}
 M^{-1}\leq v,\theta\leq M,\qquad \mathcal N_T\leq M,
 \label{eq:boot}
\end{equation}
where
\begin{align}
 \mathcal N_T={}&\sup_{0\leq t\leq T}
 \norm{(v-1,u,\theta-1,w,b)(t)}_{H^1(\Omega)}^2\notag\\
 &+\int_0^T\bigl\{
 \norm{v_x}_{L^2(\Omega)}^2
 +\norm{(u_x,\theta_x,w_x,b_x)}_{H^1(\Omega)}^2
 +\norm{(u_t,\theta_t,w_t,b_t)}_{L^2(\Omega)}^2\bigr\}\dd t.
 \label{eq:boot-norm}
\end{align}
Assume also that
\begin{equation}
 0\leq\gamma\leq1,\qquad \gamma\log M\leq\log2,
 \qquad \varepsilon_{\gamma}:=\gamma(1+M)^8\leq\varepsilon_*.
 \label{eq:small-gamma}
\end{equation}
Here $\varepsilon_*>0$ will be chosen sufficiently small, depending
only on the data, $\beta$, and fixed coefficients. In particular,
\begin{equation}
 \frac12\leq\mu\leq2,\qquad
 \mu_x=\gamma\mu\frac{\theta_x}{\theta},\qquad
 \mu_t=\gamma\mu\frac{\theta_t}{\theta}.
 \label{eq:mu-basic}
\end{equation}

For either viscosity law, multiplying \eqref{eq:mom}--\eqref{eq:ind} by $u$, $w$,
and $b$, respectively, adding the resulting identities, and using
$v_t=u_x$, we obtain the mechanical energy equation. Subtracting it
from \eqref{eq:energy} gives the temperature equation
\begin{equation}
 \theta_t+\frac{\theta}{v}u_x
 =\left(\frac{\theta^\beta\theta_x}{v}\right)_x
 +\frac{\mu u_x^2+|w_x|^2+|b_x|^2}{v}.                    \label{eq:temp}
\end{equation}

We multiply \eqref{eq:mass}, \eqref{eq:mom},
\eqref{eq:trans}, \eqref{eq:ind}, and
\eqref{eq:temp} by
\[
 1-\frac1v,\qquad u,\qquad w,\qquad b,\qquad
 1-\frac1\theta,
\]
respectively, and add the resulting identities. Using
\[
 \left(1-\frac1v\right)v_t=\Phi(v)_t,\qquad
 \left(1-\frac1\theta\right)\theta_t=\Phi(\theta)_t,
 \qquad
 b\cdot(vb)_t
 =\frac12\bigl(v|b|^2\bigr)_t+\frac12|b|^2u_x,
\]
we find that 
  the relative energy density of the MHD system  \(\eta\)  defined by  \eqref{eq:eta}  satisfies
\begin{equation}
 \eta_t-\left(\frac{\left(\mathcal K_\beta(\theta)\right)_x}{v}\right)_x+\mathcal D_\beta=(\mathcal G_\beta)_x,
                                                               \label{eq:ent}
\end{equation} with $\mathcal K_\beta(\xi)$  as in \eqref{eq:K-beta},  
\begin{equation}
 \mathcal D_\beta
 \triangleq\frac{\mu u_x^2+|w_x|^2+|b_x|^2}{v\theta}
 +\frac{\theta^{\beta-2}\theta_x^2}{v},           \label{eq:D}
\end{equation}
and
\begin{equation}
 \mathcal G_\beta
\triangleq u\left(1-\frac{\theta}{v}-\frac12|b|^2\right)
 +\frac{\mu uu_x}{v}+w\cdot b
 +\frac{w\cdot w_x+b\cdot b_x}{v}
 .
                                                               \label{eq:G}
\end{equation} 
We now record several properties of $\mathcal K_\beta$ that will be used below.
\begin{remark}\label{remark:k} 
For every $\xi>0$, the following elementary inequalities hold:
\begin{equation}
 |\xi-1|^2\leq C\Phi(\xi)\{1+\Phi(\xi)\},
 \qquad
 \xi\leq C\{1+\Phi(\xi)\}.                       \label{eq:Phi}
\end{equation}
Since $
 \mathcal K_\beta(\theta)\geq0 $ for $\theta>0$  and
$\mathcal K_\beta(\theta)\geq\mathcal K_0(\theta)=\Phi(\theta)$
for $\theta\geq1$, these inequalities imply 
\[
 \theta\leq C\{1+\mathcal K_\beta(\theta)\},
\] which, together with \eqref{eq:eta}, implies \eqref{eq:vtb}.
\end{remark}

A direct consequence of \eqref{eq:ent} is the following energy estimate.
\begin{lemma}\label{l:ent}
For every $0\leq t\leq T$,
\begin{equation}
 \int_\Omega\eta \dd x  +\int_0^t\!\int_\Omega \mathcal D_\beta\dd x\dd s
 =E_0,                                                \label{eq:ent-id}
\end{equation}
where
\begin{equation*}
 E_0=\int_\Omega\left\{
 \frac{u_0^2+|w_0|^2+v_0|b_0|^2}{2}
 +\Phi(v_0)+\Phi(\theta_0)\right\}\dd x.
\end{equation*}

\end{lemma}

\begin{proof}
For both half-line problems, the condition $\theta_x(0,t)=0$ gives
\[
 \frac{(\mathcal K_\beta(\theta))_x}{v}(0,t)=0.
\]
The boundary terms in $\mathcal G_\beta$ also vanish at $x=0$:
in the Neumann case, $u=w=0$ and $b_x=0$, while in the Dirichlet
case, $u=w=b=0$.
The boundary terms at infinity are handled by a standard cutoff argument,
using the finite-time bounds and regularity of the solution.
Integrating \eqref{eq:ent} over
$\Omega\times(0,t)$ therefore gives \eqref{eq:ent-id}.
\end{proof}

\begin{lemma}\label{l:cells}
There exist constants $0<m_0<M_0<\infty$ such that, for every
admissible integer $N$ and all $0\leq t\leq T$,
\begin{equation}
 \begin{gathered}
 m_0\leq\int_N^{N+1}v\dd x\leq M_0, \quad
 m_0\leq\int_N^{N+1}\theta\dd x\leq M_0,
 \end{gathered}                                        \label{eq:cell}
\end{equation}
and
\begin{equation}
 \bigl|\{x\in\Omega:\theta(x,t)<\frac12\}\bigr|\leq C,
 \qquad 0\leq t\leq T.                                      \label{eq:cold}
\end{equation}
Moreover, for every fixed $K>1$,
\begin{equation}
 \bigl|\{\theta>K\}\bigr|
 +\int_{\{\theta>K\}}\theta\dd x
 +\int_{\{\theta>K\}}v\dd x
 \leq C_K,\qquad 0\leq t\leq T,                   \label{eq:hot}
\end{equation}  where we set  $\{\theta>K\}\triangleq \{x\in\Omega:\theta(x,t)>K\}.$
\end{lemma}

\begin{proof} Jensen's inequality and
\eqref{eq:ent-id} give
\[
 \Phi\left(\int_N^{N+1}v\dd x\right)
 \leq\int_N^{N+1}\Phi(v)\dd x\leq E_0.
\]
Thus, the mean of $v$ lies in a fixed compact subinterval of
$(0,\infty)$ depending only on $E_0$.
The same argument applies to $\theta$ and proves \eqref{eq:cell}.
Since $\Phi(z)\geq\Phi\left(\frac12\right)>0$ for
$0<z<\frac12$, we also have
\[
 E_0\geq\int_{\{\theta<\frac12\}}\Phi(\theta)\dd x
 \geq\Phi\left(\frac12\right)
 \bigl|\{\theta<\frac12\}\bigr|,
\]
which proves \eqref{eq:cold}.

For a fixed $K>1$, we have
\[
 1+z\leq C_K\Phi(z)\quad(z>K),
 \qquad z\leq C\{1+\Phi(z)\}\quad(z>0),
\]
and hence \eqref{eq:ent-id} yields
\[
 \bigl|\{\theta>K\}\bigr|
 +\int_{\{\theta>K\}}\theta\dd x
 +\int_{\{\theta>K\}}v\dd x
 \leq C_K.
\]
This proves \eqref{eq:hot}.
\end{proof}

\begin{lemma}\label{l:osc}
There exist constants $m_\theta,C,\delta>0$ such that, for almost
every $0\leq t\leq T$ and every $x\in\Omega$, 
\begin{equation}
 \theta(x,t)\geq m_\theta 
 -C\int_\Omega
 \frac{\theta^{\beta-2}\theta_y^2}{v}\dd y.
                                                        \label{eq:osclin}
\end{equation}
Writing $D_\theta(t)=\int_\Omega\frac{\theta^{\beta-2}\theta_x^2}{v}\dd x$,
we may choose $\delta>0$ so that
\begin{equation}
 c_\theta\leq\theta(x,t)\leq C_\theta
 \quad\text{whenever }D_\theta(t)\leq\delta,
 \qquad x\in\Omega. \label{eq:good-two}
\end{equation}
\end{lemma}

\begin{proof}
By \eqref{eq:cell} and the intermediate value theorem, for each
admissible $N$, there exists $y_N(t)\in[N,N+1]$ such that
\[
 \theta(y_N(t),t)=\int_N^{N+1}\theta\dd y\in[m_0,M_0].
\]
Set
\[
 D_\theta(t)=\int_\Omega
 \frac{\theta^{\beta-2}\theta_y^2}{v}\dd y.
\]
If $\beta>0$, then, for $x\in[N,N+1]$,
\begin{align*}
 &\left|\theta^{\frac{\beta}{2}}(x,t)-
 \theta^{\frac{\beta}{2}}(y_N(t),t)\right|\\
 &\quad\leq\frac\beta2
 \left(\int_N^{N+1}
 \frac{\theta^{\beta-2}\theta_y^2}{v}\dd y\right)^{\frac12}
 \left(\int_N^{N+1}v\dd y\right)^{\frac12}\\
 &\quad\leq C_0 D_\theta(t)^{\frac12}.
\end{align*}
Choose $\delta_0>0$ such that
$C_0\delta_0^{\frac12}\leq\frac12m_0^{\frac\beta2}$.
Whenever $D_\theta(t)\leq\delta_0$, we have
\[
 \theta(x,t)
 \geq\left\{\theta(y_N(t),t)^{\frac\beta2}
 -C_0 D_\theta(t)^{\frac12}\right\}^{\frac2\beta}
 \geq\theta(y_N(t),t)-C_1 D_\theta(t)^{\frac12}
 \geq m_0-C_1 D_\theta(t)^{\frac12},
\]
where $C_1>0$. The second inequality follows from the Lipschitz
continuity of $z^{\frac2\beta}$ on a fixed compact interval bounded
away from zero.

If $\beta=0$, the same argument applied to $\log\theta$ gives
\begin{align*}
 |\log\theta(x,t)-\log\theta(y_N(t),t)|
 &\leq
 \left(\int_N^{N+1}\frac{\theta_y^2}{v\theta^2}\dd y\right)^{\frac12}
 \left(\int_N^{N+1}v\dd y\right)^{\frac12}\\
 &\leq C_0 D_\theta(t)^{\frac12}.
\end{align*}
Thus, for some $C_1>0$,
\[
 \theta(x,t)\geq\theta(y_N(t),t)
 \mathrm e^{-C_0 D_\theta(t)^{\frac12}}
 \geq m_0-C_1 D_\theta(t)^{\frac12},
\]
where we used $\mathrm e^{-z}\geq1-z$.
Consequently, there exist constants $\delta_0,C_1>0$ such that  
\[
 \theta(x,t)\geq m_0-C_1D_\theta(t)^{\frac12}\geq\frac{m_0}{2}
 -\frac{C_1^2}{2m_0}D_\theta(t) 
\]  provided $D_\theta(t)\leq\delta_0$. Thus, \eqref{eq:osclin} holds provided we set
\[m_\theta=\frac{m_0}{2}, \quad
 C=\max\left\{
 C_1,m_0\delta_0^{-\frac12},
 \frac{C_1^2}{2m_0},\frac{m_0}{2\delta_0}
 \right\}.
\]
The same estimates for $\theta^{\frac{\beta}{2}}$ or $\log\theta$ give
\eqref{eq:good-two}. Decreasing $\delta$ if necessary, we also
require $\delta\leq \frac{m_\theta}{2C}$, with $m_\theta$ and $C$ as
in \eqref{eq:osclin}.
\end{proof}

We will repeatedly use the following one-dimensional inequalities:
\begin{equation}
 \norm f_{L^\infty}^2\leq2\norm f_{L^2}\norm{f_x}_{L^2},
 \qquad
 \norm f_{L^p}\leq C_p\norm f_{L^2}^{\frac12+\frac1p}
 \norm{f_x}_{L^2}^{\frac12-\frac1p},\quad 2<p\leq\infty, \label{eq:GN}
\end{equation}
valid for $f\in H^1(\Omega)$.

\section{Uniform bounds for the specific volume and a lower bound for the temperature}\label{sec:vup}
\subsection{Auxiliary estimates}

Let $\mathcal V=H^1(\Omega)$, with dual
$\mathcal V^*=(H^1(\Omega))^*$.
\begin{lemma}\label{l:res}
Fix $0<T<T_*$. For almost every $t\in(0,T)$, there exists a unique
$\psi(t)\in H^1(\Omega)$ satisfying
\begin{equation}
 \int_\Omega\left(\frac{\psi_x\varphi_x}{v}+\psi\varphi\right)\dd x
 =\int_\Omega\eta\varphi\dd x
 \qquad\text{for every }\varphi\in H^1(\Omega).          \label{eq:reswk}
\end{equation}
Define
\begin{equation}
 F=-\frac{\psi_x}{v},
 \qquad
 \mathcal H(t)=\frac12\int_\Omega
 \left(\frac{\psi_x^2}{v}+\psi^2\right)\dd x.       \label{eq:FH}
\end{equation}
Then there exists a positive constant $C$, independent of $T$ and of the finite-time pointwise bounds for $v$, and, in case $\eqref{eq:viscosity}_2$, of $M$ and $\gamma$, such that
\begin{align}
 &\psi\geq0,\qquad
 \int_\Omega\psi\dd x=\int_\Omega\eta\dd x,
 \qquad F_x=\eta-\psi,                              \label{eq:res-L1}\\
 &\norm F_{L^\infty(\Omega)}\leq E_0,
 \qquad \norm\psi_{L^\infty(\Omega)}\leq C,       \label{eq:resinf}\\
 &\mathcal H(t)+\int_\Omega F^2\dd x\leq C.
                                                               \label{eq:resst}
\end{align}
On the whole line, $F$ has a continuous representative tending to zero
at both infinities. On the half-line, $F(0,t)=0$ and $F(x,t)\to0$
as $x\to\infty$. Moreover, $\mathcal H\in W^{1,1}(0,T)$, and the
following identity holds in $\mathcal D'(0,T)$:
\begin{equation}
 \mathcal H_t
 = \int_\Omega F\left(\mathcal K_\beta(\theta)\right)_x\dd x+\int_\Omega vF\mathcal G_\beta\dd x
 -\int_\Omega\mathcal D_\beta\psi\dd x
 +\frac12\int_\Omega u_xF^2\dd x.                 \label{eq:rest}
\end{equation}
\end{lemma}

\begin{proof}
Lemma~\ref{lem:01a} gives, on the interval under consideration,
$0<c_T\leq v\leq C_T$.
Thus,
\[
 a(\phi,\varphi;t)
 =\int_\Omega\left(\frac{\phi_x\varphi_x}{v}
 +\phi\varphi\right)\dd x
\]
defines a continuous bilinear form on $H^{1}(\Omega)$  satisfying
\[
 a(\phi,\phi;t)\geq
 \min\{1,C_T^{-1}\}\norm{\phi}_{\mathcal V}^2.
\]
The Lax--Milgram theorem therefore gives a unique solution $\psi$.
Taking $\varphi=\min\{\psi,0\}$ in
\eqref{eq:reswk}, we obtain $\psi\geq0$.

Choose cutoff functions $0\leq\chi_R\leq1$ such that $\chi_R$
increases to one and $\norm{\chi_R'}_{L^2}\to0$ as $R\to\infty$.
Test \eqref{eq:reswk} with $\chi_R$.
For the terms involving $\psi$ and $\eta$, we apply the monotone
convergence theorem and the dominated convergence theorem,
respectively. The gradient term tends to zero by the
Cauchy--Schwarz inequality, the bounds $c_T\leq v\leq C_T$, and
$\norm{\chi_R'}_{L^2}\to0$.
Letting $R\to\infty$, we obtain
\begin{equation}
 \int_\Omega\psi\dd x=\int_\Omega\eta\dd x\leq E_0. \label{eq:psi-L1}
\end{equation}
The weak equation also gives $F_x=\eta-\psi$.
On the half-line, its natural boundary condition is $F(0,t)=0$.
On the whole line, $F\in L^2(\mathbb R)$ and
$F_x\in L^1(\mathbb R)$, so the limits of $F$ at both infinities
are zero. Hence
\[
 F(x,t)=
 \begin{cases}
 \displaystyle\int_{-\infty}^x(\eta-\psi)\dd y,
     &\Omega=\mathbb R,\\[2mm]
 \displaystyle\int_0^x(\eta-\psi)\dd y,
     &\Omega=\mathbb R_+.
 \end{cases}
\]
By \eqref{eq:psi-L1},
$\int_\Omega\eta\dd x=\int_\Omega\psi\dd x\leq E_0$.
It follows that $\norm F_{L^\infty}\leq E_0$.
In the half-line case, the equality of the two integrals also gives
$F(x,t)\to0$ as $x\to\infty$.

For each admissible integer $N$,
$\int_N^{N+1}\psi\dd x\leq E_0$, so there exists
$x_N\in[N,N+1]$ such that $\psi(x_N,t)\leq E_0$.
Since $\psi_x=-vF$, for $x\in[N,N+1]$ we have
\[
 |\psi(x,t)-\psi(x_N,t)|
 \leq\norm F_{L^\infty}\int_N^{N+1}v\dd y
 \leq E_0M_0.
\]
Thus $\norm\psi_{L^\infty}\leq E_0(1+M_0)$.
Testing \eqref{eq:reswk} with $\psi$, we obtain
\[
 2\mathcal H(t)=\int_\Omega\eta\psi\dd x
 \leq E_0\norm\psi_{L^\infty}\leq C.
\]
To estimate the unweighted $L^2$ norm of $F$, we use
$\Phi(v)\geq\Phi(\frac12)$ on $\{v<\frac12\}$, so that
$|\{v<\frac12\}|\leq \frac{E_0}{\Phi(\frac12)}$. Consequently,
\[
 \int_\Omega F^2\dd x \leq2\int_\Omega vF^2\dd x +\norm F_{L^\infty}^2|\{v<\frac12\}|
 \leq4\mathcal H(t)+\frac{E_0^3}{\Phi(\frac12)}\leq C.
\]
This proves \eqref{eq:resinf}--\eqref{eq:resst}
with the stated dependence of the constants.

We next justify the time derivative in
\eqref{eq:rest}.
By \eqref{eq:finite} and the one-dimensional embedding,
\[
 (v^{-1})_t=-\frac{v_t}{v^2}\in L^2(0,T;L^\infty(\Omega)).
\]
The bounds for the strong solution on $[0,T]$ and
\eqref{eq:finite} also give
\[
\mathcal G_\beta\in L^2(0,T;L^2(\Omega)),\qquad
\frac{(\mathcal K_\beta(\theta))_x}{v}
=\frac{\theta^{\beta-1}(\theta-1)\theta_x}{v}
\in L^2(0,T;L^2(\Omega)),
\]
while \eqref{eq:ent-id} yields
\[
\mathcal D_\beta\in L^1(\Omega\times(0,T)).
\]
By \eqref{eq:ent} and the boundary conditions,
for every \(\varphi\in H^1(\Omega)\),
\[
\begin{aligned}
\langle\eta_t,\varphi\rangle
={}&-\int_\Omega
\left\{\mathcal G_\beta+
\frac{(\mathcal K_\beta(\theta))_x}{v}\right\}
\varphi_x\,\dd x
-\int_\Omega\mathcal D_\beta\varphi\,\dd x .
\end{aligned}
\]
Consequently,
\[
\|\eta_t\|_{\mathcal V^*}
\leq
\|\mathcal G_\beta\|_{L^2}
+\left\|\frac{(\mathcal K_\beta(\theta))_x}{v}\right\|_{L^2}
+C\|\mathcal D_\beta\|_{L^1},
\]
and hence
\[
v^{-1}\in W^{1,2}(0,T; L^\infty(\Omega)),
\qquad
\eta\in W^{1,1}(0,T; \mathcal V^*).
\]
Let $E_T\subset(0,T)$ be a set of full measure on which
\eqref{eq:reswk}, \eqref{eq:resst},
and \eqref{eq:ent} hold.
Fix $s,t\in E_T$ with $s<t$.
Subtracting \eqref{eq:reswk} at these two times gives
\[
 \int_\Omega\left\{ \frac{(\psi(t)-\psi(s))_x\varphi_x}{v(t)}
 +(\psi(t)-\psi(s))\varphi\right\}\dd x =\langle\eta(t)-\eta(s),\varphi\rangle
 -\int_\Omega\bigl(v^{-1}(t)-v^{-1}(s)\bigr) \psi_x(s)\varphi_x\dd x.
\]
Take $\varphi=\psi(t)-\psi(s)$.
The coercivity of $a(\phi,\varphi;t)$ on the fixed time interval gives
\begin{align*}
 C_T^{-1}\norm{\psi(t)-\psi(s)}_{\mathcal V}^2
 &\leq\left\{\norm{\eta(t)-\eta(s)}_{\mathcal V^*}
 +\norm{v^{-1}(t)-v^{-1}(s)}_{L^\infty}
 \norm{\psi_x(s)}_{L^2}\right\}
 \norm{\psi(t)-\psi(s)}_{\mathcal V}.
\end{align*}
The uniform bound for $\mathcal H$ and $v\leq C_T$ imply that
$\sup_{r\in E_T}\norm{\psi_x(r)}_{L^2}\leq C_T$.
Therefore,
\begin{align*}
 \norm{\psi(t)-\psi(s)}_{\mathcal V}
 &\leq C_T\left\{
 \norm{\eta(t)-\eta(s)}_{\mathcal V^*}
 +\norm{v^{-1}(t)-v^{-1}(s)}_{L^\infty}\right\}
 \\
 &\leq C_T\int_s^t\left\{
 \norm{\partial_r\eta(r)}_{\mathcal V^*}
 +\norm{\partial_r(v^{-1})(r)}_{L^\infty}\right\}\dd r.
\end{align*}
It follows that $t\mapsto\psi(t)$, initially defined on $E_T$,
has a unique extension to $W^{1,1}(0,T;H^1(\Omega))$, which we
continue to denote by $\psi$.
By the continuity in time of $v^{-1}$, $\eta$, and $\psi$,
\eqref{eq:reswk} holds for every $t\in[0,T]$.
The uniform bound for $\mathcal H$ extends by continuity as well.

Both sides of \eqref{eq:reswk} belong to
$W^{1,1}(0,T;\mathcal V^*)$.
Differentiating in time, we obtain, for almost every $t$ and every
$\varphi\in H^{ 1}(\Omega)$,
\[
 \int_\Omega\left(\frac{\psi_{xt}\varphi_x}{v}
 +\psi_t\varphi\right)\dd x
 =\langle\eta_t,\varphi\rangle
 +\int_\Omega\frac{u_x\psi_x\varphi_x}{v^2}\dd x.
\]
Taking $\varphi=\psi(t)$ and applying the product rule to
\eqref{eq:FH}, we obtain
\[
 \frac{\dd}{\dd t}\mathcal H =\int_\Omega\left(
 \frac{\psi_x\psi_{xt}}{v}+\psi\psi_t\right)\dd x
 -\frac12\int_\Omega\frac{u_x\psi_x^2}{v^2}\dd x =\langle\eta_t,\psi\rangle
 +\frac12\int_\Omega u_xF^2\dd x.
\]
Finally, we pair \eqref{eq:ent} with $\psi$.
On the half-line, both boundary terms at $x=0$ vanish because
\[
\theta_x(0,t)=0,\qquad \mathcal G_\beta(0,t)=0.
\]
At infinity, we pass to the limit using spatial cutoffs.
Since $\psi_x=-vF$,
\begin{align*}
 \langle\eta_t,\psi\rangle &
 =-\int_\Omega\frac{\left(\mathcal K_\beta(\theta)\right)_x}{v}\psi_x\dd x-\int_\Omega\mathcal G_\beta\psi_x\dd x
 -\int_\Omega\mathcal D_\beta\psi\dd x
 \\&=\int_\Omega F\left(\mathcal K_\beta(\theta)\right)_x\dd x+\int_\Omega vF\mathcal G_\beta\dd x
 -\int_\Omega\mathcal D_\beta\psi\dd x.
\end{align*}
This argument is valid on any fixed interval $[0,T]$.
The finite-time bounds are used to construct the elliptic solution
and justify differentiation in time, while the constants in
\eqref{eq:resinf}--\eqref{eq:resst}
are independent of $T$.
This proves \eqref{eq:rest}.
\end{proof}

Substituting \eqref{eq:G} into
\eqref{eq:rest}, we obtain
\begin{align}
 &\mathcal H_t-\int_\Omega F\bigl(\mathcal K_\beta(\theta)\bigr)_x\dd x+\int_\Omega\mathcal D_\beta\psi\dd x
 \notag\\&=\int_\Omega Fu\left(v-\theta-\frac12v|b|^2\right)\dd x
 +\int_\Omega\mu Fuu_x\dd x
 +\int_\Omega Fv\,w\cdot b\dd x\notag\\
 &\quad+\int_\Omega Fw\cdot w_x\dd x
 +\int_\Omega Fb\cdot b_x\dd x
  +\frac12\int_\Omega u_xF^2\dd x. \label{eq:res-flux}
\end{align}
Since $F_x=\eta-\psi$, integration by parts gives
\begin{align*}
 \int_\Omega F\bigl(\mathcal K_\beta(\theta)\bigr)_x\dd x
 &=-\int_\Omega(\eta-\psi)\mathcal K_\beta(\theta)\dd x,\\
 \int_\Omega Fw\cdot w_x\dd x
 &=-\frac12\int_\Omega(\eta-\psi)|w|^2\dd x,\\
 \int_\Omega Fb\cdot b_x\dd x
 &=-\frac12\int_\Omega(\eta-\psi)|b|^2\dd x.
\end{align*}
On the half-line, the boundary term at $x=0$ vanishes because
$F(0,t)=0$. At infinity, the calculation is justified by inserting
spatial cutoffs and then letting the cutoff radius tend to infinity.
By \eqref{eq:Phi},
\begin{align} \left|v-\theta-\frac12v|b|^2-\eta\right|^2&=\left| (v-1)-(\theta-1)-\frac12v|b|^2-\eta \right|^2
 \leq C\eta(1+\eta).        \label{eq:resrem}
\end{align}
Given $\varepsilon>0$, \eqref{eq:resinf} allows us
to choose $\Lambda>0$, independent of time, such that
$\frac{\norm\psi_{L^\infty}}{\Lambda}\leq\varepsilon$.
On $\{\eta<\Lambda\}$, we have $\Phi(\theta)\leq\eta<\Lambda$,
so $\theta$ lies in a compact subinterval of $(0,\infty)$ depending
only on $\Lambda$.
Hence $\mathcal K_\beta(\theta)\leq C_{\Lambda,\beta}$ there, and
\eqref{eq:res-L1} yields
\begin{equation}
\begin{aligned}
 \int_\Omega\mathcal K_\beta(\theta)\psi\dd x
 &\leq\frac{\norm\psi_{L^\infty}}{\Lambda}
 \int_{\{\eta\geq\Lambda\}}
 \mathcal K_\beta(\theta)\eta\dd x
 +C_{\Lambda,\beta}\int_{\{\eta<\Lambda\}}\psi\dd x\\
 &\leq\varepsilon\int_\Omega
 \mathcal K_\beta(\theta)\eta\dd x+C_\varepsilon.
\end{aligned}
\label{eq:Kpsi}
\end{equation}

\begin{proposition}\label{p:tlow}
Suppose that $v\geq c_1>0$ and $\mu\geq c_2>0$ on
$\Omega\times[0,T]$, where $c_1,c_2$ are independent of $T$.
For every fixed $\beta\geq0$, there exists $c>0$, depending also
on $c_1,c_2$ but independent of $T$, such that
\begin{equation}
 \theta(x,t)\geq c,
 \qquad x\in\Omega,\quad 0\leq t\leq T.                    \label{eq:t-low}
\end{equation}
\end{proposition}

\begin{proof} Set
\begin{equation}
 A_0=\max\left\{2,\frac{2}{m_\theta},
 \norm{\theta_0^{-1}}_{L^\infty(\Omega)}\right\},
 \qquad
 Z=(\theta^{-1}-A_0)_+.                              \label{eq:invcut}
\end{equation}
Fix $p>2$. Multiplying \eqref{eq:temp} by
$-\theta^{-2}Z^p$ (cf.\ \cite{CaoPengSun2021,LuShiXiong2021,TongWangZhang2026})
and integrating over $\Omega$, we use the chain
rule for the positive part and integrate the heat conduction term by
parts. The boundary terms vanish under the prescribed boundary
conditions, and we obtain
\begin{align}
 &\frac1{p+1}\frac{\dd}{\dd t}\int_\Omega Z^{p+1}\dd x
 +\int_\Omega\frac{\theta^\beta\theta_x^2}{v}
 \left(2\theta^{-3}Z^p+p\theta^{-4}Z^{p-1}\right)\dd x
 \notag\\
 &\quad+\int_\Omega
 \frac{\mu u_x^2+|w_x|^2+|b_x|^2}{v\theta^2}Z^p\dd x
 =\int_\Omega\frac{u_x}{v\theta}Z^p\dd x.          \label{eq:inv-id}
\end{align}
Young's inequality gives
\begin{align*}
 \left|\int_\Omega\frac{u_x}{v\theta}Z^p\dd x\right|
 &\leq\frac12\int_\Omega\frac{\mu u_x^2}{v\theta^2}Z^p\dd x
 +\frac1{2c_2}\int_\Omega\frac{Z^p}{v}\dd x\\
 &\leq\frac12\int_\Omega\frac{\mu u_x^2}{v\theta^2}Z^p\dd x
 +C\int_\Omega Z^p\dd x.
\end{align*}
Absorbing the first term into the left-hand side of
\eqref{eq:inv-id} and dropping the remaining
nonnegative terms, we obtain
\begin{equation}
 \frac1{p+1}\frac{\dd}{\dd t}\int_\Omega Z^{p+1}\dd x
 \leq C\int_\Omega Z^p\dd x.                       \label{eq:inv-Lp}
\end{equation}

Since $A_0\geq2$,
\[
 \{x\in\Omega:Z(x,t)>0\}
 \subset\{x\in\Omega:\theta(x,t)<\frac12\}.
\]
Thus, \eqref{eq:cold} and H\"older's inequality give
\begin{equation}
 \int_\Omega Z^p\dd x
 \leq C\left(\int_\Omega Z^{p+1}\dd x\right)^{\frac p{p+1}},
                                                               \label{eq:invH}
\end{equation}
where $C$ is independent of $p$. Define
\[
 X_p(t)=\int_\Omega Z^{p+1}(x,t)\dd x,
 \qquad
 Y_p(t)=X_p(t)^{\frac1{p+1}}
 =\norm{Z(t)}_{L^{p+1}(\Omega)}.
\]
The finite-time regularity of the strong solution and the chain rule
for the positive part imply that $X_p$ is absolutely continuous on
every finite time interval. Hence $Y_p$ is continuous, and the set
\[
 \mathcal O_p=\{0\leq t<T:Y_p(t)>0\}
\]
is open. On every compact subinterval of $\mathcal O_p$, $X_p$ is
bounded away from zero, so $Y_p$ is absolutely continuous there.
Using \eqref{eq:inv-Lp},
\eqref{eq:invH}, and the chain rule, we obtain
\begin{equation}
 Y_p'(t)
 =\frac1{p+1}X_p(t)^{-\frac p{p+1}}X_p'(t)
 \leq C                                                \label{eq:Ypdt}
\end{equation}
for almost every $t\in\mathcal O_p$, where $C$ is independent of $p$
and of the finite time interval under consideration.

With $\delta$ as in Lemma~\ref{l:osc}, set
\[
 \mathcal B  =\left\{0\leq t<T:
 \int_\Omega\frac{\theta^{\beta-2}\theta_x^2}{v}\dd x>\delta
 \right\}.
\]
By 
\eqref{eq:ent-id} and Chebyshev's inequality, we have
\[
 |\mathcal B  |\leq\frac{E_0}{\delta}.
\]  
Moreover, the choice of $\delta$ and   \eqref{eq:osclin} imply that 
for almost every $t\notin\mathcal B  $, 
\begin{equation} 
 \inf_{x\in\Omega}\theta(x,t)\geq\frac{m_\theta}{2},
                                                        \label{eq:good}
\end{equation}  
which together with \eqref{eq:invcut} implies
that $Z(\cdot,t)=0$ at these times. Consequently,
\[
 |\mathcal O_p\setminus\mathcal B  |=0.
\]

Fix $t\in[0,T)$. If $Y_p(t)=0$, there is nothing to prove.
Otherwise, let $(s,r)$ be the connected component of $\mathcal O_p$
containing $t$. The choice
$A_0\geq\norm{\theta_0^{-1}}_{L^\infty}$ in
\eqref{eq:invcut} gives $Y_p(0)=0$.
By continuity, $Y_p(s)=0$.
For $0<\varepsilon<t-s$, integrating \eqref{eq:Ypdt}
over $[s+\varepsilon,t]$ yields
\[
 Y_p(t)\leq Y_p(s+\varepsilon)+C(t-s-\varepsilon).
\]
Letting $\varepsilon\to0^+$, we obtain
\[
 Y_p(t)\leq C(t-s).
\]
Since $(s,t)\subset\mathcal O_p$ and
$|\mathcal O_p\setminus\mathcal B  |=0$,
\[
 t-s=|\mathcal B  \cap(s,t)|\leq|\mathcal B  |.
\]
It follows that
\[
 Y_p(t)\leq\frac{CE_0}{\delta},
\]
with a constant independent of $p$ and $t$.
Letting $p\to\infty$, we obtain
\[
 \norm{Z(t)}_{L^\infty(\Omega)}\leq C,
 \qquad 0\leq t<T.
\]
Finally, \eqref{eq:invcut} gives
\[
 \theta^{-1}(x,t)\leq A_0+C,
 \qquad (x,t)\in\Omega\times[0,T).
\]
Continuity at $T$ gives \eqref{eq:t-low}.
\end{proof}

\subsection{Density-dependent viscosity}\label{sec:density}

Assume $\eqref{eq:viscosity}_1$.
We next derive a uniform positive lower bound for the specific volume.
Define
\begin{equation}
 M(v)=\int_1^v\frac{\mu(z)}z\dd z
 =\mu_1\log v+\frac{\mu_2}{\alpha}(1-v^{-\alpha}),
 \qquad M'(v)=\frac{\mu(v)}v>0. \label{eq:Mv}
\end{equation}
Set
\begin{equation}
 \sigma=\frac{\mu(v)u_x}{v}-\frac{\theta}{v}-\frac12|b|^2.
 \label{eq:sigma}
\end{equation}
Then $u_t=\sigma_x$ and $M(v)_t=\sigma+\frac{\theta}{v}+\frac{1}{2}|b|^2$.

\begin{lemma}\label{l:Kaz}
For an admissible integer $N$ and $x\in[N,N+1]$, define
\begin{align}
 U_N(x,t)&=\int_N^x u(y,t)\dd y
 -\int_N^{N+1}\int_N^z u(y,t)\dd y\dd z,\notag\\
 Z_N(x,t)&=\exp\left\{\frac{M(v(x,t))-U_N(x,t)}{\mu_1}\right\},
 \qquad \overline Z_N(t)=\int_N^{N+1}Z_N(y,t)\dd y.
 \label{eq:Zdef}
\end{align}
There are positive constants independent of $N$ and $T$ such that
\begin{equation}
 |U_N|\leq C,\qquad \frac{Z_N}{v}\leq C,\qquad
 c\leq\overline Z_N\leq C,\qquad
 \int_N^{N+1}\frac{Z_N}{v}\dd x\geq c. \label{eq:Zbounds}
\end{equation}
Moreover,
\begin{equation}
 0\leq\frac1{\overline Z_N}
 \int_N^{N+1}Z_N\left(\frac{\theta}{v}+\frac12|b|^2\right)\dd y
 \leq C, \label{eq:Zmean}
\end{equation}
and
\begin{equation}
 \partial_t\log\frac{Z_N(x,t)}{\overline Z_N(t)}
 =\frac1{\mu_1}\left\{\frac{\theta}{v}+\frac12|b|^2
 -\frac1{\overline Z_N}\int_N^{N+1}
 Z_N\left(\frac{\theta}{v}+\frac12|b|^2\right)\dd y\right\}.
 \label{eq:Z-evol}
\end{equation}
\end{lemma}

\begin{proof}
The bound for $U_N$ follows from \eqref{eq:ent-id} and the unit
length of $[N,N+1]$. Integrating $u_t=\sigma_x$ over the cell gives
\[
 (U_N)_t=\sigma-\int_N^{N+1}\sigma\dd x,
 \qquad
 (M(v)-U_N)_t=\frac{\theta}{v}+\frac12|b|^2+\int_N^{N+1}\sigma\dd x.
\]
The endpoint stress cancels in the first identity, also for the cell
adjacent to $x=0$. By \eqref{eq:Zdef},
\begin{align*}
 \frac{(Z_N)_t}{Z_N}
 &=\frac1{\mu_1}\left\{\frac{\theta}{v}+\frac12|b|^2
   +\int_N^{N+1}\sigma(y,t)\dd y\right\},\\
 \frac{(\overline Z_N)_t}{\overline Z_N}
 &=\frac1{\mu_1}\left\{\frac1{\overline Z_N}
   \int_N^{N+1}Z_N\left(\frac{\theta}{v}+\frac12|b|^2\right)\dd y
   +\int_N^{N+1}\sigma(y,t)\dd y\right\}.
\end{align*}
Subtracting these identities gives \eqref{eq:Z-evol}.

To prove \eqref{eq:Zbounds}, observe that
\begin{equation}
 \exp\left\{\frac{M(z)}{\mu_1}\right\}
 =z\exp\left\{\frac{\mu_2}{\alpha\mu_1}(1-z^{-\alpha})\right\}
 \leq Cz,\qquad z>0. \label{eq:fM}
\end{equation}
For every fixed $a>0$, we also have $\exp\left\{\frac{M(z)}{\mu_1}\right\}\geq c_a z$ for
$z\geq a$. Hence $\frac{Z_N}{v}\leq C$ and $\overline Z_N\leq C$.
Since \eqref{eq:cell} implies
\[
 \int_{[N,N+1]\cap\{v\geq \frac{m_0}{2}\}}v\dd x\geq\frac{m_0}{2},
\]
we obtain $\overline Z_N\geq c$. Choose $a\in(0,1)$ so small
that $\Phi(a)>2E_0+1$. The entropy bound then gives
\[
 |[N,N+1]\cap\{v<a\}|\leq\frac{E_0}{\Phi(a)}<\frac12.
\]
On the complementary set, $\frac{Z_N}{v}\geq c_a$, which proves the
last inequality in \eqref{eq:Zbounds}. Finally,
\[
 \frac1{\overline Z_N}\int_N^{N+1}
 Z_N\left(\frac{\theta}{v}+\frac12|b|^2\right)\dd y
 \leq C\int_N^{N+1}(\theta+v|b|^2)\dd x\leq C
\]
by \eqref{eq:cell} and \eqref{eq:ent-id}.

All identities above hold for the local strong solution.
Indeed, $v_t=u_x\in L^2(0,T;H^1)$ and
$\sigma+1\in L^2(0,T;H^1)$ under the temporary state bounds.
Thus the cell functions are absolutely continuous in time with
values in $C([N,N+1])$, and the endpoint traces and differentiations
used above are well defined.
\end{proof}

\begin{proposition}\label{p:vlow}
There exists a constant $c>0$, independent of $T$, such that
\begin{equation}
 v(x,t)\geq c,\qquad x\in\Omega,\quad 0\leq t\leq T.
 \label{eq:vlow}
\end{equation}
\end{proposition}

\begin{proof}
Let $\delta$ be as in \eqref{eq:good-two} and put
\[
 \mathcal B=\{0\leq t<T_*:D_\theta(t)>\delta\}.
\]
The entropy estimate, applied on every interval of existence, gives
\begin{equation}
 |\mathcal B|+\int_{\mathcal B}D_\theta(t)\dd t\leq C.
 \label{eq:bad-measure}
\end{equation}
For fixed $N$ and $x\in [N,N+1]$, set
$R_N(t)=\log\frac{Z_N(x,t)}{\overline Z_N(t)}$.
By \eqref{eq:Zbounds}, the initial value of $R_N$ is bounded
below independently of $N,x$. If $R_N<-L$, then
$M(v(x,t))\leq C-\mu_1L$, by \eqref{eq:Zdef} and
\eqref{eq:Zbounds}. Since $M$ is strictly increasing and
$M(v)\to-\infty$ as $v\to0$, \eqref{eq:Zmean} allows us to
choose a fixed $L$ so large that, at such points,
\[
 \frac{c_\theta}{v(x,t)}
 \geq\frac1{\overline Z_N}\int_N^{N+1}
 Z_N\left(\frac{\theta}{v}+\frac12|b|^2\right)\dd y.
\]
It follows from \eqref{eq:Z-evol} and \eqref{eq:good-two} that
$R_N'\geq0$ whenever $R_N<-L$ and $t\notin\mathcal B$,
whereas $R_N'\geq-C$ at all other times.
Integrating on a connected component of
$\{R_N<\min(-L,R_N(0))\}$ therefore yields
\[
 R_N(t)\geq\min\{-L,R_N(0)\}-C|\mathcal B|\geq-C.
\]
Consequently $Z_N\geq c$ and $M(v)\geq-C$, by
\eqref{eq:Zdef} and \eqref{eq:Zbounds}.
The properties of $M$ in \eqref{eq:Mv} now give \eqref{eq:vlow}.
\end{proof}

Since $\mu\geq\mu_1$, Proposition~\ref{p:vlow} and
Proposition~\ref{p:tlow} give \eqref{eq:t-low}.

\begin{lemma}\label{l:heat-conv}
For $0\leq s<t\leq T$, we have
\begin{equation}
 \int_s^t\!\int_\Omega\frac{\mu(v)u_x^2}{v}\dd x\dd r
 \leq C+C\int_s^t\!\int_\Omega\mathcal K_\beta(\theta)\eta\dd x\dd r.
 \label{eq:heat-conv}
\end{equation}
\end{lemma}

\begin{proof}
Set $\chi(\xi)=\left(1-\frac{2}{\xi}\right)_+$ and
\[
 H(\xi)=
 \begin{cases}
 0,&\xi\leq2,\\
 \xi-2-2\log\frac{\xi}{2},&\xi>2.
 \end{cases}
\]
Then $H'=\chi$ and $0\leq H\leq C\Phi$.
Multiplying \eqref{eq:temp} by $\chi(\theta)$ and integrating
by parts, we obtain
\begin{align*}
 \int_\Omega\frac{\mu(v)u_x^2+|w_x|^2+|b_x|^2}{v}\chi(\theta)\dd x
 &=\frac{\dd}{\dd t}\int_\Omega H(\theta)\dd x
 +2\int_{\{\theta>2\}}\frac{\theta^{\beta-2}\theta_x^2}{v}\dd x\\
 &\quad+\int_\Omega\frac{\theta u_x}{v}\chi(\theta)\dd x.
\end{align*}
The pressure term satisfies
\[
 \left|\int_\Omega\frac{\theta u_x}{v}\chi(\theta)\dd x\right|
 \leq\frac12\int_\Omega\frac{\mu(v)u_x^2}{v}\chi(\theta)\dd x
 +C\int_\Omega\mathcal K_\beta(\theta)\eta\dd x.
\]
Here we used $\mu(v)\geq\mu_1$, \eqref{eq:vlow}, and
$\theta^2\leq C\Phi(\theta)^2\leq C\mathcal K_\beta(\theta)\eta$
on $\{\theta>2\}$. Integrating in time and using
\eqref{eq:ent-id}, we control the integral in \eqref{eq:heat-conv}
over $\{\theta>4\}$, where $\chi\geq\frac{1}{2}$.
On $\{\theta\leq4\}$ it is bounded directly by four times
the entropy dissipation. This proves \eqref{eq:heat-conv}.
\end{proof}
\begin{proposition}\label{p:bcap}
For all $0\leq s<t\leq T$,
\begin{equation}
 \int_s^t\!\int_\Omega v\theta|b|^2\dd x\dd r
 \leq C(1+t-s).                                      \label{eq:b-cap}
\end{equation}
\end{proposition}

\begin{proof}
Set
\[
 I_{\alpha}=\mu_2\int_\Omega v^{-\alpha}Fuu_x\dd x.
\]
Integrating by parts, we also have
\begin{align}
 \mu_1\int_\Omega Fuu_x\dd x+\frac12\int_\Omega u_xF^2\dd x
 ={}&-\frac{\mu_1}{2}\int_\Omega\eta u^2\dd x
 -\int_\Omega uF\eta\dd x\notag\\
 &+\frac{\mu_1}{2}\int_\Omega\psi u^2\dd x
 +\int_\Omega uF\psi\dd x.                         \label{eq:metric}
\end{align}
Using \eqref{eq:res-flux}, we obtain
\begin{align}
 \mathcal H_t
 &+\int_\Omega\mathcal K_\beta(\theta)\eta\dd x
 +\frac12\int_\Omega\eta\bigl(\mu_1u^2+|w|^2+|b|^2\bigr)\dd x
 +\int_\Omega\mathcal D_\beta\psi\dd x\notag\\
 &=\int_\Omega Fu
 \left(v-\theta-\frac12v|b|^2-\eta\right)\dd x
 +\int_\Omega Fv\,w\cdot b\dd x \notag\\
 &\quad+\int_\Omega\mathcal K_\beta(\theta)\psi\dd x
 +\int_\Omega uF\psi\dd x
 +\frac12\int_\Omega
 \psi\bigl(\mu_1u^2+|w|^2+|b|^2\bigr)\dd x+I_{\alpha}\notag\\&\triangleq \sum_{i=1}^5I_i+I_{\alpha}.             \label{eq:res-id}
\end{align}

Young's inequality, \eqref{eq:resrem},
and \eqref{eq:resinf}--\eqref{eq:resst}
give
\begin{align*}
 |I_1|
 &\leq\varepsilon\int_\Omega\eta u^2\dd x
 +C_\varepsilon\int_\Omega F^2
 \frac{\left|v-\theta-\frac12v|b|^2-\eta\right|^2}{\eta}
 \dd x\\
 &\leq\varepsilon\int_\Omega\eta u^2\dd x+C_\varepsilon,\\
 |I_2|
 &\leq\varepsilon\int_\Omega\eta|w|^2\dd x
 +C_\varepsilon\int_\Omega vF^2\dd x\\
 &\leq\varepsilon\int_\Omega\eta|w|^2\dd x+C_\varepsilon,\\
 |I_4|+|I_5|&\leq C.
\end{align*}
Here the quotients are defined to be zero on $\{\eta=0\}$.
For $I_2$, we used $\eta\geq\frac{v|b|^2}{2}$; for $I_4$ and
$I_5$, we used \eqref{eq:ent-id} and \eqref{eq:vlow}.

Using \eqref{eq:Kpsi} for $I_3$, choose $\varepsilon$ sufficiently small and absorb the corresponding
terms in $I_1$, $I_2$, and $I_3$. Retaining the term involving
$\eta u^2$, we obtain
\begin{equation}
 \mathcal H_t+c_1\int_\Omega\mathcal K_\beta(\theta)\eta\dd x
 +c_2\int_\Omega\eta u^2\dd x\leq C+|I_{\alpha}|,
 \qquad c_1,c_2>0. \label{eq:HK-eta}
\end{equation}
For any $\varepsilon>0$, Young's inequality gives
\begin{equation}
 |I_{\alpha}|\leq\varepsilon\int_\Omega\frac{\mu(v)u_x^2}{v}\dd x
 +C_\varepsilon\int_\Omega
 \frac{v^{1-2\alpha}}{\mu(v)}F^2u^2\dd x. \label{eq:Ialpha-est}
\end{equation}
If $\alpha\geq\frac{1}{2}$, the last integral is bounded by a constant,
by \eqref{eq:vlow}, \eqref{eq:resinf}, and \eqref{eq:ent-id}.
If $0<\alpha<\frac{1}{2}$, then for any $\delta>0$,
\[
 z^{1-2\alpha}\leq\delta\Phi(z)+C_\delta\quad(z\geq c).
\]
Since $\mu(v)\geq\mu_1$ and $F$ is bounded, this implies
\[
 \int_\Omega\frac{v^{1-2\alpha}}{\mu(v)}F^2u^2\dd x
 \leq C\delta\int_\Omega\eta u^2\dd x+C_\delta.
\]
Integrate \eqref{eq:HK-eta} from $s$ to $t$ and use
\eqref{eq:Ialpha-est} and \eqref{eq:heat-conv}.
First choose $\varepsilon$ small enough to absorb the resulting
multiple of $\int_s^t\!\int_\Omega\mathcal K_\beta(\theta)\eta$.
Then, for $0<\alpha<\frac{1}{2}$, choose $\delta$ small enough to absorb
the integral of $\eta u^2$. Using $0\leq\mathcal H\leq C$, we obtain
\begin{equation}
 \int_s^t\!\int_\Omega
 \{\mathcal K_\beta(\theta)\eta+\eta u^2\}\dd x\dd r
 \leq C(1+t-s). \label{eq:Keta-int}
\end{equation}
Together with \eqref{eq:vtb} and
$\int_\Omega\eta\dd x\leq E_0$, this proves
\eqref{eq:b-cap}.
\end{proof}
\begin{lemma}\label{l:bP}
For all $0\leq s<t\leq T$, every admissible integer $N$, and every
$x\in[N,N+1]$,
\begin{equation}
 \int_s^t\left|
 \frac12|b(x,r)|^2-\frac12\int_N^{N+1}|b(y,r)|^2\dd y
 \right|\dd r
 \leq C\sqrt{1+t-s}.                                \label{eq:bP}
\end{equation}
In fact,
\begin{equation}
 \int_s^t\!\int_N^{N+1}|b||b_y|\dd y\dd r
 \leq C\sqrt{1+t-s}. \label{eq:b-osc}
\end{equation}
\end{lemma}

\begin{proof}
Fix $r$. Since $[N,N+1]$ has unit length, for $x\in[N,N+1]$ we have
\begin{align*}
 &\left|\frac12|b(x,r)|^2
 -\frac12\int_N^{N+1}|b(y,r)|^2\dd y\right|
 \\ &\quad\leq\int_N^{N+1}|b||b_y|\dd y\\
 &\quad\leq\left(\int_N^{N+1}v\theta|b|^2\dd y\right)^{\frac12}
 \left(\int_N^{N+1}\frac{|b_y|^2}{v\theta}\dd y\right)^{\frac12}.
\end{align*}
Integrating the last inequality in time and using \eqref{eq:b-cap}
and \eqref{eq:ent-id}, we obtain
\begin{align*}
& \int_s^t\!\int_N^{N+1}|b||b_y|\dd y\dd r
 \\ &\leq\left(\int_s^t\!\int_\Omega v\theta|b|^2\dd x\dd r\right)^{\frac12}
 \left(\int_s^t\!\int_\Omega
 \frac{|b_x|^2}{v\theta}\dd x\dd r\right)^{\frac12}\\
 &\leq C\sqrt{1+t-s}.
\end{align*}
This proves \eqref{eq:b-osc}; the preceding pointwise estimate
then gives \eqref{eq:bP}.
\end{proof}

\begin{proposition}\label{p:v-up}
There exists a constant $C$, independent of $T$, such that
\begin{equation}
 v(x,t)\leq C,\qquad x\in\Omega,\quad 0\leq t\leq T.
 \label{eq:v-up}
\end{equation}
\end{proposition}

\begin{proof}
Let $\mathcal B$ be the set used in Proposition~\ref{p:vlow}.
For an admissible $N$, let
\[
 L_N(t)=\max_{x\in [N,N+1]}\log Z_N(x,t)-\log\overline Z_N(t),
\]
and choose a maximizing point $x_+(t)\in [N,N+1]$.
By \eqref{eq:vlow}, \eqref{eq:fM}, and \eqref{eq:Zbounds},
\begin{equation}
 c v\leq Z_N\leq Cv,\qquad
 \max_{[N,N+1]}v\leq C v(x_+,t),\qquad v(x_+,t)\geq c.
 \label{eq:Zmax-v}
\end{equation}
Starting at a point where $\theta$ equals its cell average, we obtain
\begin{align*}
 \sqrt{\theta(x_+,t)}
 &\leq C+\frac12\int_N^{N+1}\frac{|\theta_x|}{\sqrt\theta}\dd x\\
 &\leq C+C\left(\int_N^{N+1}\frac{\theta_x^2}{v\theta^2}\dd x\right)^{\frac12}
             \left(\int_N^{N+1}v\theta\dd x\right)^{\frac12}\\
 &\leq C+C\left\{D_\theta(t)\max_{[N,N+1]}v\right\}^{\frac12}.
\end{align*}
Here we used \eqref{eq:cell} and the uniform lower bound
\eqref{eq:t-low}. Thus \eqref{eq:Zmax-v} implies
\begin{equation}
 \frac{\theta(x_+,t)}{v(x_+,t)}\leq C(1+D_\theta(t)).
 \label{eq:max-pressure}
\end{equation}

For almost every $t\notin\mathcal B$, \eqref{eq:good-two} and
\eqref{eq:Zbounds} give
\[
 \frac1{\overline Z_N}\int_N^{N+1}\frac{Z_N\theta}{v}\dd y
 \geq c_0>0,
 \qquad
 \frac{\theta(x_+,t)}{v(x_+,t)}\leq C\mathrm e^{-L_N(t)}.
\]
Choose a fixed $L_*>0$ so that the last quantity is at most
$\frac{c_0}{2}$ when $L_N>L_*$. Moreover, the fundamental theorem of
calculus and \eqref{eq:Zdef} give, for every $x\in[N,N+1]$,
\[
 \frac12|b(x,t)|^2
 -\frac1{2\overline Z_N(t)}\int_N^{N+1}Z_N(y,t)|b(y,t)|^2\dd y
 \leq\int_N^{N+1}|b||b_y|\dd y.
\]
The function $L_N$ is absolutely
continuous on $[0,T]$, by the time regularity in Lemma~\ref{l:Kaz}.
At almost every differentiability time its derivative is bounded by
the largest time derivative among the maximizing points.
Using \eqref{eq:Z-evol}, we therefore obtain, whenever $L_N>L_*$,
\[
 L_N'(t)\leq
 \begin{cases}
 -c+\mu_1^{-1}\displaystyle\int_N^{N+1}|b||b_y|\dd y,&t\notin\mathcal B,\\[2mm]
 C(1+D_\theta(t))+\mu_1^{-1}\displaystyle\int_N^{N+1}|b||b_y|\dd y,&t\in\mathcal B.
 \end{cases}
\]
In the second case we used \eqref{eq:max-pressure}.

If $L_N(t)>\max\{L_*,L_N(0)\}$, let $s$ be the left endpoint
of the corresponding connected component. Then
$L_N(s)=\max\{L_*,L_N(0)\}\leq C$. Integrating over
$(s,t)\setminus\mathcal B$ and $(s,t)\cap\mathcal B$, and using
\eqref{eq:bad-measure}, we obtain
\begin{align}
 L_N(t)&\leq L_N(s)-c|(s,t)\setminus\mathcal B|
 +C\int_{(s,t)\cap\mathcal B}(1+D_\theta(r))\dd r\notag\\
 &\quad+\frac1{\mu_1}\int_s^t\!\int_N^{N+1}|b||b_y|\dd y\dd r\notag\\
 &\leq C-c(t-s)
 +\frac1{\mu_1}\int_s^t\!\int_N^{N+1}|b||b_y|\dd y\dd r.
 \label{jhoa1}
\end{align}
Applying \eqref{eq:b-osc}, we find
\[
 L_N(t)\leq C+\sup_{\tau\geq0}
 \{-c\tau+C\sqrt{1+\tau}\}\leq C.
\]
The same bound is immediate outside these components, since
$L_N(0)\leq C$. Hence $Z_N\leq C$, and
\eqref{eq:Zmax-v} gives the required upper bound for $v$.
\end{proof}

\subsection{Temperature-dependent viscosity}\label{sec:temperature}

Assume $\eqref{eq:viscosity}_2$ and the temporary bounds
\eqref{eq:boot}--\eqref{eq:small-gamma}.
Dividing \eqref{eq:mom} by $\mu$ (cf.\ \cite{SongZhao2025}), we obtain
\begin{equation}
 \left(\frac u\mu\right)_t+h=\sigma_x,
 \qquad \sigma=\frac{u_x}{v}-\frac{\theta}{\mu v}
                         -\frac{|b|^2}{2\mu},
 \label{eq:sigma-T}
\end{equation}
where
\begin{equation}
 h=\frac{u\mu_t}{\mu^2}
   +\left(\frac\theta v+\frac{|b|^2}{2}\right)\frac{\mu_x}{\mu^2}
   -\frac{\mu_xu_x}{\mu v}.
 \label{eq:hmu-T}
\end{equation}
For every admissible $N$ and $0\leq s<t\leq T$, the temporary bounds
\eqref{eq:boot} and \eqref{eq:mu-basic} yield
\begin{align}
 \int_s^t\!\int_N^{N+1}|h|\dd x\dd r
 &\leq C\gamma\bigl\{M^{\frac{3}{2}}\sqrt{t-s}
                    +M^{\frac{7}{2}}\sqrt{t-s}+M^3\bigr\}
 \leq C\varepsilon_{\gamma}(1+t-s),\label{eq:hmu-bd-T}\\
 \int_s^t\!\int_\Omega|\mu_x|u^2\dd x\dd r
 &\leq C\gamma M^2\sqrt{t-s}
 \leq C\varepsilon_{\gamma}(1+t-s),\label{eq:mu-u2-T}\\
 \int_s^t\!\int_N^{N+1}|\mu_x||b|^2\dd x\dd r
 &\leq C\gamma M^{\frac{5}{2}}\sqrt{t-s}
 \leq C\varepsilon_{\gamma}(1+t-s).\label{eq:mu-b2-T}
\end{align}
Indeed, $|\mu_t|+|\mu_x|\leq C\gamma M(|\theta_t|+|\theta_x|)$,
$\frac{\theta}{v}+\frac{|b|^2}{2}\leq CM^2$, and the $L^2$ time-space norms of
$\theta_t,\theta_x,u_x$ are bounded by $M^{\frac{1}{2}}$.
Cauchy--Schwarz on $[N,N+1]\times[s,t]$ gives the three terms in
\eqref{eq:hmu-bd-T}. For \eqref{eq:mu-u2-T}, use
$\norm{u^2}_{L^2}\leq\norm u_{L^\infty}\norm u_{L^2}\leq CM^{\frac{1}{2}}$,
where the last $L^2$ bound follows from \eqref{eq:ent-id}.
For \eqref{eq:mu-b2-T}, use $\norm b_{L^\infty}^2\leq CM$.

\begin{lemma}\label{l:Kaz-T}
For an admissible integer $N$ and $x\in[N,N+1]$, define
\begin{align}
 B_N(x,t)
 &=v_0(x)\exp\biggl\{\int_N^x\left(\frac u\mu-\frac{u_0}{\mu(\theta_0)}\right)\dd y
 +\int_0^t\!\int_N^x h(y,r)\dd y\dd r\biggr\}, \label{eq:BN-T}\\
 Y_N(t)&=\exp\left\{\int_0^t\sigma(N,s)\dd s\right\}. \label{eq:YN-T}
\end{align}
Then
\begin{equation}
 \begin{aligned}
 v(x,t)
  =B_N(x,t)Y_N(t)\biggl[1+\int_0^t
 \frac{\theta+\frac12v|b|^2}{\mu B_N(x,s)Y_N(s)}\dd s\biggr].
  \end{aligned}                                         \label{eq:Kaz-T}
\end{equation}
Moreover, there exist constants $C>0$ and $a>0$ such that, for every
admissible integer $N$, every $x\in[N,N+1]$, and all $0\leq s<t\leq T$,
\begin{equation}
 C^{-1}\mathrm e^{-C\varepsilon_{\gamma}(1+t-s)}\leq\frac{B_N(x,t)}{B_N(x,s)}\leq C\mathrm e^{C\varepsilon_{\gamma}(1+t-s)},                           \label{eq:BNbd-T}
\end{equation}
\begin{equation}
 \frac{Y_N(t)}{Y_N(s)}\leq C\mathrm e^{-a(t-s)}.      \label{eq:YNdec-T}
\end{equation}
\end{lemma}

\begin{proof}
Fix $x\in[N,N+1]$. Integrating \eqref{eq:sigma-T} first from $N$ to $x$
and then from $0$ to $t$, we obtain
\begin{align*}
 \int_N^x\left(\frac u\mu-\frac{u_0}{\mu(\theta_0)}\right)\dd y+\int_0^t\!\int_N^x h\dd y\dd s
 &=\int_0^t[\sigma(x,s)-\sigma(N,s)]\dd s\\
 &=\log\frac{v(x,t)}{v_0(x)}
 -\int_0^t\left(\frac{\theta}{\mu v}+\frac{|b|^2}{2\mu}\right)\dd s
 -\int_0^t\sigma(N,s)\dd s,
\end{align*}
where we used $(\log v)_t=\frac{u_x}{v}$.
By \eqref{eq:BN-T}--\eqref{eq:YN-T},
\[
 \frac{v(x,t)}{B_N(x,t)Y_N(t)}
 =\exp\left\{\int_0^t
 \left(\frac{\theta}{\mu v}+\frac{|b|^2}{2\mu}\right)\dd s\right\}.
\]
Differentiating this identity in time, we obtain
\begin{equation}\label{vav1-T}
 \frac{\dd}{\dd t}\left(\frac{v}{B_NY_N}\right)
 =\frac{\theta}{\mu B_NY_N}
 +\frac{|b|^2}{2\mu}\frac{v}{B_NY_N},\qquad
 \frac{v(x,0)}{B_N(x,0)Y_N(0)}=1.
\end{equation}
Integration in time gives \eqref{eq:Kaz-T}.

The assumptions on the initial data and the one-dimensional Sobolev
embedding imply that $v_0$ is bounded above and away from zero.
For $x\in[N,N+1]$, \eqref{eq:ent-id} gives
\begin{align*}
 \left|\int_N^x\frac u\mu\dd y\right|
 &\leq2|x-N|^{\frac12}\left(\int_N^{N+1}u^2\dd y\right)^{\frac12}
 \leq2(2E_0)^{\frac12}.
\end{align*}
The term involving $\frac{u_0}{\mu(\theta_0)}$ is estimated in the same way. Combining these bounds with \eqref{eq:hmu-bd-T} proves \eqref{eq:BNbd-T}.

To prove \eqref{eq:YNdec-T}, integrate
$\sigma(N,t)=\sigma(x,t)-\int_N^x\{\left(\frac u\mu\right)_t+h\}\dd y$
over $x\in(N,N+1)$. Young's inequality then gives
\begin{align}
 \sigma(N,t)
 &=\int_N^{N+1}\left(\frac{u_x}{v}-\frac{\theta}{\mu v}
                -\frac{|b|^2}{2\mu}\right)\dd x
 -H_N'(t)-\int_N^{N+1}(N+1-y)h\dd y\notag\\
 &\leq-\frac12\int_N^{N+1}\frac{\theta}{\mu v}\dd x
       +C\int_N^{N+1}\frac{\mu u_x^2}{v\theta}\dd x
       -\int_N^{N+1}\frac{|b|^2}{2\mu}\dd x
       -H_N'(t)+\int_N^{N+1}|h|\dd x,\label{eq:lj02-T}
\end{align}
where
\[
 H_N(t)=\int_N^{N+1}(N+1-y)\frac u\mu\dd y
\]
satisfies, by \eqref{eq:ent-id} and \eqref{eq:mu-basic},
\begin{equation}
 |H_N(t)|\leq C\norm{u(t)}_{L^2(N,N+1)}\leq C.
 \label{eq:lj01-T}
\end{equation}

By  Lemmas \ref{l:cells} and \ref{l:osc},
\begin{align}
 \int_N^{N+1}\frac{\theta}{v}\dd x& \notag
 \geq\inf_{x\in [N,N+1]}\theta(x,t)\int_N^{N+1}\frac1v\dd x\\&
 \geq \inf_{x\in [N,N+1]}\theta(x,t)
 \left(\int_N^{N+1}v\dd x\right)^{-1} \notag
\\& \geq c_0-C\int_\Omega
 \frac{\theta^{\beta-2}\theta_x^2}{v}\dd x, \label{eq:P-cell-T}
\end{align}
which, together with \eqref{eq:mu-basic} and \eqref{eq:lj02-T}, gives
\begin{align}
 \sigma(N,t)\leq{}&-a_0
 +C\int_\Omega\mathcal D_\beta\dd x
 -\int_N^{N+1}\frac{|b|^2}{2\mu}\dd x-H_N'(t)
 +\int_N^{N+1}|h|\dd x,
 \label{eq:sigavg-T}
\end{align}
for some $a_0>0$ independent of $M$. Integrating over $[s,t]$ and
using \eqref{eq:ent-id}, \eqref{eq:hmu-bd-T}, and \eqref{eq:lj01-T}, we find
\[
 \int_s^t\sigma(N,r)\dd r
 \leq C-a_0(t-s)-\int_s^t\!\int_N^{N+1}\frac{|b|^2}{2\mu}\dd x\dd r
       +C\varepsilon_{\gamma}(1+t-s).
\]
Choose $\varepsilon_*$ in \eqref{eq:small-gamma} so that the last
linear coefficient is at most $\frac{a_0}{2}$. Exponentiation yields
\begin{equation}
 \frac{Y_N(t)}{Y_N(s)}\leq C\exp\left\{-a(t-s)
 -\int_s^t\!\int_N^{N+1}\frac{|b|^2}{2\mu}\dd x\dd r\right\},
 \qquad a=\frac{a_0}{2},
 \label{eqyn1-T}
\end{equation}
which implies \eqref{eq:YNdec-T}.
\end{proof}

The identities above give the following estimate for $v$.
\begin{lemma}\label{l:Kaza-T} There exist constants $C>0$ and $a>0$ such that, for every
admissible integer $N$ and every $(x,t)\in[N,N+1]\times[0,T]$,   
\begin{align} \label{jhoa1-T} v(x,t)\le  & C\exp\left\{-a t+C\varepsilon_{\gamma}(1+t)+\frac12\int_0^t\left|\frac{|b(x,r)|^2}{\mu(x,r)}
 -\int_{ N}^{N+1}\frac{|b|^2}{\mu}\dd y \right|\dd r\right\}\notag\\&+C\int_0^t\exp\left\{-a(t-s)+C\varepsilon_{\gamma}(1+t-s)+\frac12\int_s^t\left|\frac{|b(x,r)|^2}{\mu(x,r)}
 -\int_{ N}^{N+1}\frac{|b|^2}{\mu}\dd y \right|\dd r\right\}\theta \dd s. \end{align}   \end{lemma}
\begin{proof}
By \eqref{vav1-T}, we have
\[
 \frac{\dd}{\dd t}\left[
 \exp\left\{-\frac12\int_0^t\frac{|b(x,r)|^2}{\mu(x,r)}\dd r\right\}
 \frac{v}{B_NY_N}\right]
 =\exp\left\{-\frac12\int_0^t\frac{|b(x,r)|^2}{\mu(x,r)}\dd r\right\}
 \frac{\theta}{\mu B_NY_N},
\]
which gives
 \begin{align}v(x,t)&=B_N(x,t)Y_N(t)
 \exp\left\{\frac12\int_0^t\frac{|b(x,r)|^2}{\mu(x,r)}\dd r\right\}\notag\\
 &\quad+\int_0^t\frac{B_N(x,t)Y_N(t)}{B_N(x,s)Y_N(s)}
 \exp\left\{\frac12\int_s^t\frac{|b(x,r)|^2}{\mu(x,r)}\dd r\right\}
 \frac{\theta(x,s)}{\mu(x,s)}\dd s.  \label{eqyn2-T}\end{align}

By \eqref{eqyn1-T},
for $x\in [N,N+1]$ and $0\leq s<t\leq T$, 
\begin{align*}
 &\frac{Y_N(t)}{Y_N(s)}
 \exp\left\{\frac12\int_s^t\frac{|b(x,r)|^2}{\mu(x,r)}\dd r\right\}\\
 & \leq C\exp\left\{-a(t-s)+C\varepsilon_{\gamma}(1+t-s)+\frac12\int_s^t\left|\frac{|b(x,r)|^2}{\mu(x,r)}
 -\int_{N}^{N+1}\frac{|b|^2}{\mu}\dd y \right|\dd r\right\} , 
\end{align*}  which together with \eqref{eq:BNbd-T}, \eqref{eq:mu-basic}, and \eqref{eqyn2-T}   gives \eqref{jhoa1-T}. 
\end{proof}

\begin{proposition}\label{p:bcap-T}
For all $0\leq s<t\leq T$,
\begin{equation}
 \int_s^t\!\int_\Omega v\theta|b|^2\dd x\dd r
 \leq C(1+t-s).                                      \label{eq:b-cap-T}
\end{equation}
\end{proposition}

\begin{proof}
The two terms involving $u_x$ satisfy
\begin{align}
 \int_\Omega \mu Fuu_x\dd x+\frac12\int_\Omega u_xF^2\dd x
 ={}&-\frac12\int_\Omega\mu\eta u^2\dd x
 -\int_\Omega uF\eta\dd x\notag\\
 &+\frac12\int_\Omega\mu\psi u^2\dd x
 +\int_\Omega uF\psi\dd x-\frac12\int_\Omega\mu_xFu^2\dd x. \label{eq:metric-T}
\end{align}
Using \eqref{eq:res-flux}, we obtain
\begin{align}
 \mathcal H_t
 &+\int_\Omega\mathcal K_\beta(\theta)\eta\dd x
 +\frac12\int_\Omega\eta\bigl(\mu u^2+|w|^2+|b|^2\bigr)\dd x
 +\int_\Omega\mathcal D_\beta\psi\dd x\notag\\
 &=\int_\Omega Fu
 \left(v-\theta-\frac12v|b|^2-\eta\right)\dd x
 +\int_\Omega Fv\,w\cdot b\dd x \notag\\
 &\quad+\int_\Omega\mathcal K_\beta(\theta)\psi\dd x
 +\int_\Omega uF\psi\dd x
 +\frac12\int_\Omega
 \psi\bigl(\mu u^2+|w|^2+|b|^2\bigr)\dd x\notag\\&\quad-\frac12\int_\Omega\mu_xFu^2\dd x\notag\\&\triangleq \sum_{i=1}^5I_i-\frac12\int_\Omega\mu_xFu^2\dd x.             \label{eq:res-id-T}
\end{align}

Young's inequality, \eqref{eq:resrem},
and \eqref{eq:resinf}--\eqref{eq:resst}
give
\begin{align*}
 |I_1|
 &\leq\varepsilon\int_\Omega\mu\eta u^2\dd x
 +C_\varepsilon\int_\Omega F^2
 \frac{\left|v-\theta-\frac12v|b|^2-\eta\right|^2}{\eta}
 \dd x\\
 &\leq\varepsilon\int_\Omega\mu\eta u^2\dd x+C_\varepsilon,\\
 |I_2|
 &\leq\varepsilon\int_\Omega\eta|w|^2\dd x
 +C_\varepsilon\int_\Omega vF^2\dd x\\
 &\leq\varepsilon\int_\Omega\eta|w|^2\dd x+C_\varepsilon,\\
 |I_4|&\leq C,\\ |I_5|&\leq\varepsilon\int_\Omega\eta|b|^2\dd x+C_\varepsilon.
\end{align*}
Here the quotients are defined to be zero on $\{\eta=0\}$.
For $I_2$, we used $\eta\geq\frac{v|b|^2}{2}$; for $I_4$, use \eqref{eq:ent-id} and \eqref{eq:resinf}--\eqref{eq:resst}.
For the magnetic part of $I_5$, choose $\Lambda$ so large that
$\frac{\norm\psi_{L^\infty}}{\Lambda}\leq2\varepsilon$. On $\{\eta<\Lambda\}$,
$\Phi(v)\leq\Lambda$ implies $v\geq c_\Lambda>0$, whereas on
$\{\eta\geq\Lambda\}$ the integral is absorbed by the $\eta|b|^2$ term.
Thus
\[
 \int_\Omega\psi|b|^2\dd x
 \leq2\varepsilon\int_\Omega\eta|b|^2\dd x
    +C_\varepsilon\int_\Omega v|b|^2\dd x.
\]
The parts involving $u,w$ follow directly from \eqref{eq:ent-id}.
Also, \eqref{eq:resinf} gives
\[
 \left|\frac12\int_\Omega\mu_xFu^2\dd x\right|
 \leq C\int_\Omega|\mu_x|u^2\dd x.
\]

Using \eqref{eq:Kpsi} for $I_3$, choose $\varepsilon$ sufficiently small and absorb the corresponding
terms in $I_1$, $I_2$, $I_3$, and $I_5$ into the left-hand side of
\eqref{eq:res-id-T}.
Dropping the remaining nonnegative terms, we obtain
\begin{equation}
 \mathcal H_t+c\int_\Omega
 \mathcal K_\beta(\theta)\eta\dd x\leq C+C\int_\Omega|\mu_x|u^2\dd x.          \label{eq:HK-eta-T}
\end{equation}

Integrating \eqref{eq:HK-eta-T} from $s$ to $t$ and using
$0\leq\mathcal H\leq C$ and \eqref{eq:mu-u2-T}, we obtain
\[
 \int_s^t\!\int_\Omega\mathcal K_\beta(\theta)\eta\dd x\dd r
 \leq C+C(t-s)+C\varepsilon_{\gamma}(1+t-s)\leq C(1+t-s).
\]
Together with \eqref{eq:vtb} and
$\int_\Omega\eta\dd x\leq E_0$, this proves
\eqref{eq:b-cap-T}.
\end{proof}
By the Cauchy--Schwarz inequality, \eqref{eq:b-cap-T}, and
\eqref{eq:ent-id}, for every admissible integer $N$ we have
\begin{align*}
 &\int_s^t\!\int_N^{N+1}|b||b_y|\dd y\dd r\\
 &\quad\leq\left(\int_s^t\!\int_\Omega v\theta|b|^2\dd x\dd r\right)^{\frac12}
 \left(\int_s^t\!\int_\Omega\frac{|b_x|^2}{v\theta}\dd x\dd r\right)^{\frac12}
 \leq C\sqrt{1+t-s}.
\end{align*}
Since
\[
 \left(\frac{|b|^2}{2\mu}\right)_x
 =\frac{b\cdot b_x}{\mu}-\frac{|b|^2\mu_x}{2\mu^2},
\]
integration in space and time, together with \eqref{eq:mu-basic} and
\eqref{eq:mu-b2-T}, gives
\begin{equation}
 \int_s^t\left|\frac{|b(x,r)|^2}{2\mu(x,r)} -\int_N^{N+1}\frac{|b(y,r)|^2}{2\mu(y,r)}\dd
 y\right|\dd r \leq C\sqrt{1+t-s}+C\varepsilon_{\gamma}(1+t-s).
 \label{eq:bP-mu-T}
\end{equation}

\begin{proposition}\label{p:v-up-T}
For every fixed $\beta\geq0$, there exists a constant $C>0 $ 
 such that for any $(x,t)\in\Omega\times [0,T]$ 
\begin{equation}
 v(x,t)\leq C.
                 \label{eq:v-up-T}
\end{equation}
\end{proposition}

\begin{proof} 
Substituting   \eqref{eq:bP-mu-T} into
\eqref{jhoa1-T}, choosing $\varepsilon_*$ smaller if necessary, and
using $C\sqrt{1+\tau}\leq \frac{a\tau}{4}+C$, we obtain, for every
admissible integer $N$ and every $(x,t)\in[N,N+1]\times[0,T]$,  
\begin{align}
 v(x,t) \leq C+C\int_0^t\mathrm e^{-\frac{a(t-s)}{2}}\theta(x,s)\dd s.
                                                               \label{eq:vuconv-T}
\end{align}

For every $x\in[N,N+1]$, choose $y_N(t)$ as in the proof of
Lemma~\ref{l:osc}. The fundamental theorem of calculus gives
\begin{align*}
 \theta(x,t)^{\frac{\beta+1}{2}}
 &\leq\theta(y_N(t),t)^{\frac{\beta+1}{2}}
       +\frac{\beta+1}{2}\int_N^{N+1}
                         \theta^{\frac{\beta-1}{2}}|\theta_y|\dd y\\
 &\leq C+C\left(\int_\Omega
                \frac{\theta^{\beta-2}\theta_y^2}{v}\dd y\right)^{\frac{1}{2}}
             \left(\int_N^{N+1}v\theta\dd y\right)^{\frac{1}{2}}\\
 &\leq C+C\left\{D_\theta(t)
                    \max_{y\in[N,N+1]}v(y,t)\right\}^{\frac{1}{2}}.
\end{align*}
Since $\beta\geq0$,
\[
 \theta(x,t)\leq C+CD_\theta(t)\max_{y\in [N,N+1]}v(y,t).
\]
Substituting this into \eqref{eq:vuconv-T} and
taking the maximum over $[N,N+1]$, we obtain
\[
 \max_{x\in [N,N+1]}v(x,t)
 \leq C+C\int_0^tD_\theta(s)\max_{y\in [N,N+1]}v(y,s)\dd s.
\]
Gronwall's inequality and \eqref{eq:ent-id} yield
\[
 \max_{x\in [N,N+1]}v(x,t)
 \leq C\exp\left\{C\int_0^tD_\theta(s)\dd s\right\}\leq C.
\]
The constant is independent of $N$ and $t$.
Since $N$ is arbitrary, \eqref{eq:v-up-T} follows.
\end{proof}
\begin{proposition}\label{p:vlow-T}
There exists a constant $c>0$, independent of time, such that
\begin{equation}
 v(x,t)\geq c,
 \qquad x\in\Omega,\quad 0\leq t\leq T.                    \label{eq:vlow-T}
\end{equation}
\end{proposition}

\begin{proof}
For an admissible $N$, set
\[
 V_N(t)=\int_N^{N+1}v\dd x,\qquad
 U_N(x,t)=\int_N^x\frac u\mu\dd y,\qquad
 A_N(t)=\frac1{V_N(t)}\int_N^{N+1}vU_N\dd x.
\]
By \eqref{eq:cell}, \eqref{eq:mu-basic}, and \eqref{eq:ent-id},
\[
 m_0\leq V_N(t)\leq M_0,\qquad
 \norm{U_N(t)}_{L^\infty(N,N+1)}+|A_N(t)|\leq C.
\]
Equation~\eqref{eq:sigma-T} gives
\[
 (U_N)_t=\sigma(x,t)-\sigma(N,t)-\int_N^x h\dd y.
\]
Differentiating $V_NA_N=\int_N^{N+1}vU_N\dd x$
(cf.\ \cite{HuangShiSun2021,SongZhao2025}) and using $v_t=u_x$,
we obtain
\begin{align}
 \sigma(N,t)={}&\frac{V_N'}{V_N}-A_N'
 -\frac1{V_N}\int_N^{N+1}\frac{\theta+\frac{v|b|^2}{2}}{\mu}\dd x\notag\\
 &+\frac1{V_N}\int_N^{N+1}u_x(U_N-A_N)\dd x
 -\frac1{V_N}\int_N^{N+1}v(x,t)\int_N^x h(y,t)\dd y\dd x.
 \label{eq:weighted-cell-T}
\end{align}
The pressure integral is bounded by \eqref{eq:cell} and
\eqref{eq:ent-id}. For the fourth term, \eqref{eq:v-up-T} yields
\[
 \int_s^t\!\int_N^{N+1}|u_x|\dd x\dd r \leq\left(\int_s^t\!\int_N^{N+1} \frac{\mu
 u_x^2}{v\theta}\dd x\dd r\right)^{\frac{1}{2}} \left(\int_s^t\!\int_N^{N+1}
 \frac{v\theta}{\mu}\dd x\dd r\right)^{\frac{1}{2}} \leq C\sqrt{t-s}.
\]
The last term is controlled by \eqref{eq:hmu-bd-T}.
Integrating \eqref{eq:weighted-cell-T} from $s$ to $t$, we therefore have
\begin{align*}
 \int_s^t\sigma(N,r)\dd r
 &\geq\log\frac{V_N(t)}{V_N(s)}-A_N(t)+A_N(s)
       -C(t-s)-C\sqrt{t-s}-C\varepsilon_{\gamma}(1+t-s)\\
 &\geq-C(1+t-s).
\end{align*}
Together with \eqref{eq:BNbd-T}, this gives
\begin{equation}
 \frac{B_N(x,t)Y_N(t)}{B_N(x,s)Y_N(s)}
 \geq c\mathrm e^{-C(t-s)},\qquad 0\leq s<t\leq T.
 \label{eq:BY-lower-T}
\end{equation}
The constants in \eqref{eq:BY-lower-T} are independent of the
pointwise lower bound for $v$ in \eqref{eq:boot}.

Choose $\delta>0$ so small that \eqref{eq:osclin} gives
$\theta(x,r)\geq \frac{m_\theta}{2}$ whenever $D_\theta(r)\leq\delta$,
and set
\[
 \mathcal B=\{r\in[0,T]:D_\theta(r)>\delta\},
 \qquad L=1+\frac{E_0}{\delta}.
\]
By \eqref{eq:ent-id}, $|\mathcal B|\leq \frac{E_0}{\delta}$.
The representation formula \eqref{eq:Kaz-T}, started at any
$s\in[0,t]$, reads
\[
 v(x,t)={}\frac{B_N(x,t)Y_N(t)}{B_N(x,s)Y_N(s)}v(x,s)
 +\int_s^t\frac{B_N(x,t)Y_N(t)}{B_N(x,r)Y_N(r)} \frac{\theta+\frac{v|b|^2}{2}}{\mu}(x,r)\dd r.
\]
If $t\geq L$, the set $[t-L,t]\setminus\mathcal B$ has measure at
least one. Taking $s=t-L$, dropping the nonnegative initial and
magnetic terms, and using \eqref{eq:BY-lower-T}, we find
\[
 v(x,t)\geq c\int_{[t-L,t]\setminus\mathcal B}
                   \mathrm e^{-C(t-r)}\theta(x,r)\dd r
          \geq c m_\theta\mathrm e^{-CL}.
\]
For $0\leq t<L$, the initial term with $s=0$ gives
$v(x,t)\geq c\mathrm e^{-CL}\inf_\Omega v_0>0$.
This proves \eqref{eq:vlow-T}.
\end{proof}

By \eqref{eq:mu-basic} and Proposition~\ref{p:vlow-T},
Proposition~\ref{p:tlow} gives \eqref{eq:t-low}.

\section{Uniform estimates for higher derivatives}\label{sec:high}

By Propositions~\ref{p:vlow}, \ref{p:v-up}, and \ref{p:vlow-T},
\ref{p:v-up-T}, as applicable, and Proposition~\ref{p:tlow}, we have
\begin{equation}
 c\leq v\leq C,\qquad \theta\geq c,\qquad c\leq\mu\leq C
 \quad\hbox{on }\Omega\times[0,T].                    \label{eq:state}
\end{equation}
For $\eqref{eq:viscosity}_1$, $|\mu'(v)|\leq C$ also follows.
For $\eqref{eq:viscosity}_2$, the bounds on $\mu$ use \eqref{eq:mu-basic}.

\begin{proposition}\label{p:cap}
For every fixed $\beta\geq0$,
\begin{equation}
\sup_{0\leq t\leq T}\int_\Omega
 \{u^2+|w|^2+ |b|^2 \}\dd x+ \int_0^T\!\int_\Omega
 \{\mu u_x^2+|w_x|^2+|b_x|^2\}\dd x\dd t\leq C.
                                                               \label{eq:cap}
\end{equation}
If $\beta>0$, then also
\begin{align}
 &\int_0^T\!\int_\Omega
 \frac{\theta_x^2}{\theta}\dd x\dd t\leq C.
                                                               \label{eq:cappos}
\end{align}
\end{proposition}

\begin{proof}
Decreasing $c$ if necessary, we may
assume that $c\leq1$ in \eqref{eq:state}. Fix $K\geq4$, to be chosen below, and define
on $[c,\infty)$
\begin{equation}
 g_\beta(\xi)=
 \begin{cases}
 \displaystyle 2(1-\xi^{-1}),&c\leq\xi\leq K,\\[2mm]
 \displaystyle
 2-\frac1K-
 \frac{\beta\left(\frac{\xi}{K}\right)^{-\beta}
 +\left(\frac{\xi}{K}\right)^{-\frac12}}
 {K(1+\beta)},&\xi>K.
 \end{cases}                                         \label{eq:capdef}
\end{equation}
Extend $g_\beta$ to $(0,c)$ as a bounded, nondecreasing Lipschitz
function. Then $g_\beta$ is bounded, continuous, and nondecreasing,
with $g_\beta(1)=0$. Set
\[
 F_\beta(\xi)=\int_1^\xi g_\beta(s)\dd s,
 \qquad \zeta_\beta=2-g_\beta,
\] so that, for every $\xi\geq c$,
\begin{equation} \label{eq:lj03} 0\leq F_\beta(\xi)\leq C\Phi(\xi),\qquad
 \zeta_\beta(\xi)\geq K^{-1}.
 \end{equation}

Next, multiplying \eqref{eq:mom}--\eqref{eq:ind} by $u$, $w$,
and $b$, respectively, and adding the resulting identities to
\eqref{eq:mass} multiplied by $1-v^{-1}$, we obtain
\begin{equation}
 \frac{\dd}{\dd t}\int_\Omega\left\{
 \frac{u^2+|w|^2+v|b|^2}{2}+\Phi(v)\right\}\dd x
 +\int_\Omega\frac{\mu u_x^2+|w_x|^2+|b_x|^2}{v}\dd x
 =\int_\Omega\frac{(\theta-1)u_x}{v}\dd x.           \label{eq:mech-v}
\end{equation}
Setting $\tau\triangleq\frac{\theta}{K},$ 
multiplying \eqref{eq:temp} by $g_\beta(\theta)$, adding
twice \eqref{eq:mech-v}, and integrating by parts, we obtain
\begin{align}
 &\frac{\dd}{\dd t}\int_\Omega
 \{u^2+|w|^2+v|b|^2+2\Phi(v)+F_\beta(\theta)\}\dd x\notag\\
 &\quad+\int_\Omega\frac{
 \zeta_\beta(\theta)(\mu u_x^2+|w_x|^2+|b_x|^2)
 +g_\beta'(\theta)\theta^\beta\theta_x^2}{v}\dd x\notag\\
 &=\int_\Omega
 \frac{\{2(\theta-1)-g_\beta(\theta)\theta\}u_x}{v}\dd x\notag\\&\leq\frac12\int_\Omega
 \frac{\zeta_\beta(\theta)\mu u_x^2}{v}\dd x
 +C_\beta K\int_\Omega(\tau-1)_+^2 \dd x,
                                                               \label{eq:cap-id}
\end{align}
where the last inequality follows from \eqref{eq:lj03} and
$\left|2(\theta-1)-g_\beta(\theta)\theta\right|
\leq C_\beta(\tau-1)_+$ for $\theta\geq c$.
The constants in the following estimates have the dependence specified
in Section~\ref{sec:lower} and are independent of $K$ and time.

For the second term on the right-hand side of \eqref{eq:cap-id}, we have
\begin{align}
 C_\beta K\int_\Omega(\tau-1)_+^2 \dd x
 &\leq C_\beta K
 \norm{(\sqrt\tau-1)_+}_{L^\infty(\Omega)}^2
 \int_{\{\theta>K\}} \tau  \dd x\notag\\&\leq C_\beta
 \norm{(\sqrt\tau-1)_+}_{L^\infty(\Omega)}^2
 \notag\\ &\leq C_\beta\left(\int_{\{\theta>K\}} \tau^{\frac14}
  |(\tau^{\frac14})_x|  \dd x\right)^2\notag\\ &\leq C_\beta \int_{\{\theta>K\}} \tau^{\frac12}\dd x \int_{\{\theta>K\}}
 |(\tau^{\frac14})_x|^2 \dd x\notag\\ &\leq C_\beta \int_{\{\theta>4\}} \tau\dd x \int_{\{\theta>K\}}
 \tau^{-\frac32} \tau_x^2 \dd x \notag\\
 &\leq\frac{C_\beta}{K}
 \int_{\{\theta>K\}}
 \tau^{\beta-\frac32}\frac{\tau_x^2}{v}\dd x\notag\\
 &\leq \frac{C_\beta}{K}\int_\Omega
 \frac{g_\beta'(\theta)\theta^\beta\theta_x^2}{v}\dd x,
                                                               \label{eq:hotcap}
\end{align}
where, in the last inequality, we used the following identity on
\(\{\theta>K\}\):
\begin{equation*}
 \frac{g_\beta'(\theta)\theta^\beta\theta_x^2}{v}
 =\frac{K^\beta}{1+\beta}
 \left\{\beta^2\tau^{-1}
 +\frac12\tau^{\beta-\frac32}\right\}
 \frac{\tau_x^2}{v}.
\end{equation*}

Choosing $K=\max\{4C_\beta,4\}$, substituting
\eqref{eq:hotcap} into \eqref{eq:cap-id},
absorbing the resulting term, and integrating in time, we obtain
\begin{align*}
 &\sup_{0\leq t\leq T}\int_\Omega
 \{u^2+|w|^2+v|b|^2+2\Phi(v)+F_\beta(\theta)\}\dd x\\
 &\quad+\int_0^T\!\int_\Omega
 \frac{\zeta_\beta(\theta)(\mu u_x^2+|w_x|^2+|b_x|^2)
 +g_\beta'(\theta)\theta^\beta\theta_x^2}{v}\dd x\dd t
 \leq C.
\end{align*}
Since $\zeta_\beta(\theta)\geq K^{-1}$ and $v\leq C$,
\eqref{eq:cap} follows.

For $\beta>0$ and $c\leq\theta\leq K$, we have
\[
 g_\beta'(\theta)\theta^\beta
 =2\theta^{\beta-2}
 \geq\frac{c_{\beta,K}}{\theta}.
\]
For $\theta>K$, we use $\tau=\frac{\theta}{K}$ to obtain
\[
 g_\beta'(\theta)\theta^\beta
 =\frac{K^{\beta-2}}{1+\beta}
 \left\{\beta^2\tau^{-1}
 +\frac12\tau^{\beta-\frac32}\right\}
 \geq\frac{\beta^2K^{\beta-1}}{1+\beta}\frac1\theta.
\]
Therefore,
\[
 g_\beta'(\theta)\theta^\beta
 \geq\frac{c_{\beta,K}}{\theta},
 \qquad \theta\geq c.
\]
Together with the uniform upper and lower bounds for $v$,
this proves \eqref{eq:cappos}.
\end{proof}

\begin{lemma}\label{l:flux}
We have
\begin{align}
 &\sup_{0\leq t\leq T}
 (\|w\|_{H^1}^2
 + \|b\|_{H^1}^2 )  +\int_0^T\!\int_\Omega\biggl\{
 \left|w_t \right|^2
 + \left|b_t \right|^2
 +\left|\left(\frac{w_x}{v}\right)_x\right|^2
 +\left|\left(\frac{b_x}{v}\right)_x\right|^2
 \biggr\}\dd x\dd t\leq C.                         \label{eq:trflux}
\end{align}
Moreover,
\begin{align}
 &\int_0^T\bigl\{
 \norm{w_x(t)}_{L^\infty}^2+\norm{b_x(t)}_{L^\infty}^2
 +\norm{w_x(t)}_{L^\infty}^4+\norm{b_x(t)}_{L^\infty}^4
  \bigr\}\dd t\leq C.
                                                               \label{eq:tr-max}
\end{align}
\end{lemma}

\begin{proof}
Set
\[
 r=\frac{w_x}{v},
 \qquad s=\frac{b_x}{v}.
\]
We rewrite \eqref{eq:trans} and \eqref{eq:ind} as
\begin{align}
 r_x&=(w_t-ur)+ur-vs,\notag\\
 s_x&=v\left\{(b_t-us)+us+\frac{u_x}{v}b-r\right\}.
                                                               \label{eq:rs-id}
\end{align}

Take the inner product of the first identity in
\eqref{eq:rs-id} with $w_t-ur$ and integrate over $\Omega$.
Since $w_x=vr$ and $v_t=u_x$, integration by parts gives
\begin{align*}
 \int_\Omega r_x\cdot(w_t-ur)\dd x
 &=-\int_\Omega r\cdot(w_x)_t\dd x
 +\frac12\int_\Omega u_x|r|^2\dd x\\
 &=-\int_\Omega v r\cdot r_t\dd x
 -\frac12\int_\Omega u_x|r|^2\dd x\\
 &=-\frac12\frac{\dd}{\dd t}\int_\Omega v|r|^2\dd x.
\end{align*}
Thus,
\begin{equation}
 \frac12\frac{\dd}{\dd t}\int_\Omega v|r|^2\dd x
 +\int_\Omega|w_t-ur|^2\dd x
 =\int_\Omega(vs-ur)\cdot(w_t-ur)\dd x.             \label{eq:rE}
\end{equation}
Similarly, taking the inner product of the second identity in
\eqref{eq:rs-id} with $b_t-us$ and using $b_x=vs$ and
$v_t=u_x$, we obtain
\begin{equation}
 \frac12\frac{\dd}{\dd t}\int_\Omega v|s|^2\dd x
 +\int_\Omega v|b_t-us|^2\dd x
 =-\int_\Omega v\left(us+\frac{u_x}{v}b-r\right)
 \cdot(b_t-us)\dd x.                                  \label{eq:sE}
\end{equation}
In the half-line cases, the boundary terms above vanish since $u=w=0$
and either $b_x=0$ or $b=0$ at $x=0$.
Define
\begin{equation}
 Y(t)=\frac12\int_\Omega v
 (|r|^2+|s|^2)\dd x.                                  \label{eq:tr-Y}
\end{equation}
Adding \eqref{eq:rE} and \eqref{eq:sE}, we use the uniform
upper and lower bounds for $v$ and Young's inequality to absorb the
corresponding dissipation terms and obtain
\begin{align}
 Y'(t)&+c\int_\Omega\{|w_t-ur|^2+v|b_t-us|^2\}\dd x\notag\\
 &\leq C
\int_\Omega\left(
 |u|^2(|r|^2+|s|^2)+|u_x|^2|b|^2+|r|^2+|s|^2 \right)
 \dd x\notag\\
 &\leq C Y(t)+C\norm{u(t)}_{L^\infty}^2Y(t)
 +C\norm{b(t)}_{L^\infty}^2
 \int_\Omega\frac{u_x^2}{v}\dd x \notag\\
 &\leq C\bigl(Y(t)+\|u_x(t)\|_{L^2}^2\bigr)+CY(t)\|u_x(t)\|_{L^2}^2.
                                                               \label{eq:flux}
\end{align}

By \eqref{eq:cap} and the uniform bounds for $v$,
\begin{equation}
 \int_0^T\left( Y(t)+\|u_x\|_{L^2}^2\right)\dd t\leq C.                \label{eq:fluxL1}
\end{equation}
Applying Gronwall's inequality to \eqref{eq:flux}, we obtain
\begin{equation}\sup_{0\leq t\leq T}Y(t)+
 \int_0^T\!\int_\Omega
 \{|w_t-ur|^2+v|b_t-us|^2\}\dd x\dd t\leq C.
                                                               \label{eq:trt}
\end{equation}
Combining this with \eqref{eq:flux} and \eqref{eq:fluxL1}
also gives
\begin{align*}
 \int_0^T\!
\int_\Omega\left(
 |u|^2(|r|^2+|s|^2)+|u_x|^2|b|^2+|r|^2+|s|^2 \right)
 \dd x\dd t  \leq C .
\end{align*}
Substituting this estimate and \eqref{eq:trt}
into \eqref{eq:rs-id}, we obtain
\[
 \int_0^T\!\int_\Omega(|w_t|^2+|b_t|^2+|r_x|^2+|s_x|^2)\dd x\dd t\leq C.
\]
Together with \eqref{eq:tr-Y} and \eqref{eq:cap},
this proves \eqref{eq:trflux}.

Finally, the one-dimensional inequality
$\norm{f}_{L^\infty}^2\leq2\norm{f}_{L^2}\norm{f_x}_{L^2}$,
\eqref{eq:fluxL1}, \eqref{eq:trt}, and the
bounds for $r_x,s_x$ above give
\begin{align*}
 \int_0^T\{\norm{r(t)}_{L^\infty}^2+\norm{s(t)}_{L^\infty}^2
 +\norm{r(t)}_{L^\infty}^4+\norm{s(t)}_{L^\infty}^4\}\dd t\leq C.
\end{align*}
Since $w_x=vr$, $b_x=vs$, and $v\leq C$, we obtain the four terms
involving $w_x$ and $b_x$ in \eqref{eq:tr-max}.

Moreover, by \eqref{eq:GN}, \eqref{eq:cap}, and
\eqref{eq:trflux}, we obtain
\begin{equation}
\begin{aligned}
 \int_0^T\norm{b(t)}_{L^\infty}^4\dd t
 &\leq C\sup_{0\leq t\leq T}\norm{b(t)}_{L^2}^2
 \int_0^T\norm{b_x(t)}_{L^2}^2\dd t\leq C,\\
 \int_0^T\!\int_\Omega |b|^2|b_x|^2\dd x\dd t
 &\leq \sup_{0\leq t\leq T}\norm{b(t)}_{L^\infty}^2
 \int_0^T\norm{b_x(t)}_{L^2}^2\dd t\leq C.
\end{aligned}                                             \label{eq:b-inf}\qedhere
\end{equation}
\end{proof}

For $\beta=0$, we need the following additional estimate.

\begin{proposition}\label{p:hot0}
If $\beta=0$, then
\begin{align}
 &\sup_{0\leq t\leq T}\int_\Omega(\theta-2)_+^2\dd x
 +\int_0^T\!\int_\Omega\theta_x^2\dd x\dd t+\int_0^T\norm{(\theta(t)-2)_+}_{L^\infty}^2\dd t\leq C.
                                                               \label{eq:hot-0}
\end{align}
\end{proposition}

\begin{proof}
Multiplying \eqref{eq:temp} by $(\theta-2)_+$
(cf.\ \cite{LiLiang2016}) and integrating
over $\Omega$, we obtain
\begin{align}
 &\frac12\frac{\dd}{\dd t}\int_\Omega(\theta-2)_+^2\dd x
 +\int_{\{\theta>2\}}\frac{\theta_x^2}{v}\dd x\notag\\
 &\quad=-\int_\Omega\frac{\theta u_x}{v}(\theta-2)_+\dd x
 +\int_\Omega\frac{\mu u_x^2}{v}(\theta-2)_+\dd x
 +\int_\Omega\frac{(\theta-2)_+(|w_x|^2+|b_x|^2)}{v}\dd x.
                                                        \label{eq:heat-0}
\end{align}
Next, multiply \eqref{eq:mom} by $2u(\theta-2)_+$ and integrate
over $\Omega$. Since
\begin{align*}
 \partial_t(\theta-2)_+
 &=\mathbf1_{\{\theta>2\}}\theta_t,\\
 2u(\theta-2)_+u_t
 &=\bigl[u^2(\theta-2)_+\bigr]_t
 -\mathbf1_{\{\theta>2\}}u^2\theta_t
\end{align*}
hold almost everywhere, integration by parts yields
\begin{align*}
 &\frac{\dd}{\dd t}\int_\Omega u^2(\theta-2)_+\dd x
 +2\int_\Omega\frac{\mu(\theta-2)_+u_x^2}{v}\dd x\\
 &=\int_{\{\theta>2\}}u^2\theta_t\dd x
 -2\int_{\{\theta>2\}}\frac{\mu uu_x\theta_x}{v}\dd x\\
 &\quad+2\int_\Omega\frac{\theta u_x}{v}(\theta-2)_+\dd x
 +2\int_{\{\theta>2\}}\frac{\theta u\theta_x}{v}\dd x\\
 &\quad+2\int_\Omega\frac{|b|^2}{2}(\theta-2)_+u_x\dd x
 +2\int_{\{\theta>2\}}\frac{|b|^2}{2}u\theta_x\dd x\\
 &\triangleq \sum_{i=1}^6 I_i(t).
\end{align*}
We first estimate $I_1$. Since $\beta=0$, \eqref{eq:temp} reads
\[
 \theta_t=\left(\frac{\theta_x}{v}\right)_x
 -\frac{\theta u_x}{v}
 +\frac{\mu u_x^2+|w_x|^2+|b_x|^2}{v}.
\]
For $\rho>0$, define
\[
 \varphi_\rho(\xi)=
 \begin{cases}
  0,&\xi\leq2,\\[1mm]
  \frac{\xi-2}{\rho},&2<\xi<2+\rho,\\[2mm]
  1,&\xi\geq2+\rho.
 \end{cases}
\]
Then, for every $\varepsilon>0$,
\begin{align*}
 I_1(t)
 &=\lim_{\rho\to0^+}\int_\Omega\varphi_\rho(\theta)u^2
 \left(\frac{\theta_x}{v}\right)_x\dd x
 -\int_{\{\theta>2\}}\frac{u^2\theta u_x}{v}\dd x+\int_{\{\theta>2\}}
 \frac{u^2(\mu u_x^2+|w_x|^2+|b_x|^2)}{v}\dd x\\
 &=\lim_{\rho\to0^+}\Biggl\{-2\int_\Omega\varphi_\rho(\theta)
 \frac{uu_x\theta_x}{v}\dd x-\int_\Omega\varphi_\rho'(\theta)
 \frac{u^2\theta_x^2}{v}\dd x
\Biggr\}\\
&\quad-\int_{\{\theta>2\}}\frac{u^2\theta u_x}{v}\dd x
 +\int_{\{\theta>2\}}
 \frac{u^2(\mu u_x^2+|w_x|^2+|b_x|^2)}{v}\dd x\\
 &\leq-2\int_{\{\theta>2\}}\frac{uu_x\theta_x}{v}\dd x
 -\int_{\{\theta>2\}}\frac{u^2\theta u_x}{v}\dd x+\int_{\{\theta>2\}}
 \frac{u^2(\mu u_x^2+|w_x|^2+|b_x|^2)}{v}\dd x\\
 &\leq\varepsilon\int_\Omega\theta_x^2\dd x
 +C_\varepsilon\int_\Omega u^2u_x^2\dd x
 +C_\varepsilon
 \norm{\left(\theta(t)-\frac32\right)_+}_{L^\infty}^2+\int_{\{\theta>2\}}
 \frac{u^2(|w_x|^2+|b_x|^2)}{v}\dd x.
\end{align*}
For the first inequality above, we used $\varphi_\rho'\geq0$,
$\varphi_\rho(\theta)\to\mathbf1_{\{\theta>2\}}$, and the dominated
convergence theorem. For the last inequality, we used
\[
 \theta\leq4\left(\theta-\frac32\right)_+
 \qquad\text{on }\{\theta>2\},
\]
together with \eqref{eq:hot}, \eqref{eq:state},
\eqref{eq:ent-id}, and Young's inequality. The remaining terms satisfy
\begin{align*}
 |I_2(t)|
 &\leq\varepsilon\int_\Omega\theta_x^2\dd x
 +C_\varepsilon\int_\Omega u^2u_x^2\dd x,\\
 |I_3(t)|
 &\leq\frac12\int_\Omega\frac{\mu(\theta-2)_+u_x^2}{v}\dd x
 +C\norm{\left(\theta(t)-\frac32\right)_+}_{L^\infty}^2,\\
 |I_4(t)|
 &\leq\varepsilon\int_\Omega\theta_x^2\dd x
 +C_\varepsilon
 \norm{\left(\theta(t)-\frac32\right)_+}_{L^\infty}^2.
\end{align*}
The terms $I_5$ and $I_6$ will be estimated after time integration.

Finally, multiply \eqref{eq:mom} by $u^3$ to obtain
\begin{align}
 &\frac14\frac{\dd}{\dd t}\int_\Omega u^4\dd x
 +3\int_\Omega\frac{\mu u^2u_x^2}{v}\dd x=3\int_\Omega\frac{\theta}{v}u^2u_x\dd x
 +3\int_\Omega\frac{|b|^2}{2}u^2u_x\dd x.        \label{eq:u3-0}
\end{align}
Using $\int_\Omega u^2u_x\dd x=0$, we estimate the first term on
the right-hand side of \eqref{eq:u3-0} as follows:
\begin{align}
 \left|\int_\Omega\frac{\theta}{v}u^2u_x\dd x\right|
 &=\left|\int_\Omega\left(\frac{\theta}{v}-1\right)
 u^2u_x\dd x\right|\notag\\
 &\leq\left|\int_\Omega\left\{
 \frac{1-v}{v}+\mathbf1_{\{\theta\leq2\}}\frac{\theta-1}{v}
 \right\}u^2u_x\dd x\right|+\left|\int_{\{\theta>2\}}
 \frac{\theta-1}{v}u^2u_x\dd x\right|\notag\\
 &\leq\norm{\frac{1-v}{v}
 +\mathbf1_{\{\theta\leq2\}}\frac{\theta-1}{v}}_{L^2}
 \norm{u(t)}_{L^\infty}^2\norm{u_x(t)}_{L^2}+\left|\int_{\{\theta>2\}}
 \frac{\theta-1}{v}u^2u_x\dd x\right|\notag\\
 &\leq C\norm{u_x(t)}_{L^2}^2
 +\varepsilon\int_\Omega u^2u_x^2\dd x
 +C_\varepsilon
 \norm{\left(\theta(t)-\frac32\right)_+}_{L^\infty}^2.
                                                        \label{eq:u3-P-0}
\end{align}
Here the first inequality uses
$\frac{\theta}{v}-1=\frac{1-v}{v}+\frac{\theta-1}{v}$. The last one
follows from Hölder's and Young's inequalities,
\eqref{eq:state}, \eqref{eq:ent-id}, and
\eqref{eq:GN}, together with
\[
 \begin{gathered}
  (v-1)^2\leq C\Phi(v),\qquad
  (\theta-1)^2\mathbf1_{\{\theta\leq2\}}\leq C\Phi(\theta),\quad
 ( \theta-1)\mathbf1_{\{\theta>2\}}\leq4\left(\theta-\frac32\right)_+
  .
 \end{gathered}
\]

We add \eqref{eq:heat-0}, the identity defining
$I_1,\ldots,I_6$, and $\Lambda$ times \eqref{eq:u3-0}.
Using the estimates for $I_1,\ldots,I_4$,
\eqref{eq:u3-P-0}, \eqref{eq:cap}, and
\eqref{eq:ent-id}, and choosing $\varepsilon>0$ sufficiently
small and $\Lambda>0$ sufficiently large, we obtain
\begin{align*}
 &\int_\Omega\left\{\frac12(\theta-2)_+^2
 +u^2(\theta-2)_++\frac{\Lambda}{4}u^4\right\}\dd x
 +c\int_0^t\!\int_\Omega
 \{(\theta+u^2)u_x^2+\theta_x^2\}\dd x\dd s\\
 &\leq C+C\int_0^t
 \norm{\left(\theta(s)-\frac32\right)_+}_{L^\infty}^2\dd s+C\int_0^t
 \int_\Omega\frac{(\theta-2)_+(|w_x|^2+|b_x|^2)}{v}\dd x\dd s\\&\quad
 +C\int_0^t \int_{\{\theta>2\}}\frac{u^2(|w_x|^2+|b_x|^2)}{v}\dd x\dd s+C\int_0^t  \int_\Omega\frac{|b|^2}{2}
 (\theta-2)_+|u_x|\dd x  \dd s\\&
 \quad+C\int_0^t  \int_{\{\theta>2\}}\frac{|b|^2}{2}|u\theta_x|\dd x\dd s
 +C\int_0^t  \int_\Omega\frac{|b|^2}{2}u^2|u_x|\dd x   \dd s\\
 & \triangleq C+C\int_0^t
 \norm{\left(\theta(s)-\frac32\right)_+}_{L^\infty}^2\dd s
 +  \sum_{i=1}^5J_i(t) .
\end{align*}
Then, for every $\varepsilon>0$,
\begin{align*}
 J_1(t)+ J_2(t)
 & \leq C\int_0^t  \{\norm{w_x(s)}_{L^\infty}^2+\norm{b_x(s)}_{L^\infty}^2\}\dd s, \\
 J_3(t)& \leq\varepsilon\int_0^t \int_\Omega
 \frac{(\theta-2)_+u_x^2}{v}\dd x\dd s
 +C_\varepsilon \int_0^t \norm{b(s)}_{L^\infty}^4\dd s,\\
 J_4(t)
 & \leq\varepsilon\int_0^t \int_\Omega\theta_x^2\dd x\dd s
 +C_\varepsilon\int_0^t \norm{b(s)}_{L^\infty}^4\dd s,\\
 J_5(t)
 & \leq\varepsilon\int_0^t \int_\Omega u^2u_x^2\dd x\dd s
 +C_\varepsilon\int_0^t \norm{b(s)}_{L^\infty}^4\dd s.
\end{align*}
The estimates for $J_1$ and $J_2$ use \eqref{eq:hot},
\eqref{eq:state}, and \eqref{eq:ent-id}. For
$J_3,J_4$, and $J_5$, we additionally use Young's inequality and
\[
 \int_\Omega(\theta-2)_+\dd x+\norm{u(t)}_{L^2}^2\leq C.
\]
Moreover, \eqref{eq:tr-max} and
\eqref{eq:b-inf} give
\[ \int_0^T \left(
 \norm{w_x(t)}_{L^\infty}^2+\norm{b_x(t)}_{L^\infty}^2
 +\norm{b(t)}_{L^\infty}^4\right)\dd t\le C.
\]
Substituting the estimates for $J_1,\ldots,J_5$ and absorbing the
terms with $\varepsilon$, we obtain
\begin{align}
 &\int_\Omega\left\{\frac12(\theta-2)_+^2
 +u^2(\theta-2)_++\frac{\Lambda}{4}u^4\right\}\dd x
 +c\int_0^t\!\int_\Omega
 \{(\theta+u^2)u_x^2+\theta_x^2\}\dd x\dd s\notag\\ &\leq C+C\int_0^t
 \norm{\left(\theta(s)-\frac32\right)_+}_{L^\infty}^2\dd s\notag\\&\leq C+C\int_0^t
 \left(\int_{\{\theta>\frac32\}}|\theta_x|\dd x\right)^2\dd s\notag\\&\leq C+C\int_0^t
 \left(\int_{\{\theta>\frac32\}}v\theta\dd x\right)
 \left(\int_{\{\theta>\frac32\}}\frac{\theta_x^2}{v\theta}\dd x\right)\dd s\notag\\ &\leq  C+   C\delta\int_0^t\!\int_\Omega\theta_x^2\dd x\dd s+C_\delta\int_0^t
 \int_{\{\theta>\frac32\}}\frac{\theta_x^2}{v\theta^2}\dd x \dd s.
                                                        \label{eq:hotid0}
\end{align}
Choosing $\delta>0$ sufficiently small and using
\eqref{eq:ent-id}, we obtain \eqref{eq:hot-0}.
\end{proof}

We can now estimate the first spatial derivative of the specific volume.

\begin{lemma}\label{l:vx}
For every fixed $\beta\geq0$,
\begin{equation}
 \sup_{0\leq t\leq T}\int_\Omega v_x^2\dd x
 +\int_0^T\!\int_\Omega v_x^2\dd x\dd t\leq C.
                                                               \label{eq:vx}
\end{equation}
\end{lemma}

\begin{proof}
Estimates \eqref{eq:state}, \eqref{eq:cap},
and \eqref{eq:b-inf}, together with
\eqref{eq:cappos} for $\beta>0$ and
\eqref{eq:hot-0} for $\beta=0$, give
\begin{align}
 \int_0^T\!\int_\Omega\biggl\{&
 \mu u_x^2+|w_x|^2+|b_x|^2
 +\frac{\theta_x^2}{ \theta}+|b|^2|b_x|^2\biggr\}
 \dd x\dd t\leq C.                                  \label{eq:prehi}
\end{align}
By \eqref{eq:hot}, \eqref{eq:state}, and the Cauchy--Schwarz inequality,
\[
 \norm{(\theta(t)-2)_+}_{L^\infty}^2
 \leq\left(\int_{\{\theta>2\}}|\theta_x|\dd x\right)^2
 \leq C\int_\Omega\frac{\theta_x^2}{\theta}\dd x.
\]
Integrating in time and using \eqref{eq:prehi}, we obtain
\begin{equation}
 \int_0^T
 \norm{(\theta(t)-2)_+}_{L^\infty}^2 \dd t\leq C.
                                                        \label{eq:excurs}
\end{equation}

In case $\eqref{eq:viscosity}_1$, set
$q=M(v)_x=\frac{\mu v_x}{v}$. Using $u_x=v_t$ in \eqref{eq:mom}, we obtain
\begin{align*}
 M(v)_{xt} -u_t&=\left(\frac{\theta}{v}+\frac12|b|^2\right)_x\\&=-\frac{\theta v_x}{v^2}+\frac{\theta_x}{v}+ b\cdot b_x.
\end{align*}
Multiply this identity by $q$ and integrate over $\Omega$.
Using
\[
 -\int_\Omega u\,q_t\,\dd x
 =-\int_\Omega u\left(\frac{\mu u_x}{v}\right)_x\,\dd x
 =\int_\Omega\frac{\mu u_x^2}{v}\,\dd x
\]
we obtain
\begin{align}
 &\frac{\dd}{\dd t}\int_\Omega
 \left\{\frac12|q|^2-uq\right\}\dd x
 +\int_\Omega\frac{\theta|q|^2}{\mu v}\dd x\notag\\&=\int_\Omega\left\{
 \frac{\mu u_x^2}{v}+\frac{\theta_x q}{v}
 +q\,b\cdot b_x\right\}\dd x\\
 &\leq\frac12\int_\Omega
 \frac{\theta|q|^2}{\mu v}\dd x
 +C\int_\Omega\left\{u_x^2+\frac{\theta_x^2}{v\theta}
 +|b|^2|b_x|^2\right\}\dd x, \label{eq:lj09}
\end{align}
where the last inequality follows from \eqref{eq:state}
and Young's inequality.

Finally, by \eqref{eq:ent-id},
\[
 \int_\Omega\left\{\frac12|q|^2
 -uq\right\}\dd x
 \geq\frac14\int_\Omega|q|^2\dd x
 -C\int_\Omega u^2\dd x
 \geq\frac14\int_\Omega|q|^2\dd x-C.
\]
Integrating \eqref{eq:lj09} over $(0,t)$ and using this lower bound,
\eqref{eq:init}, \eqref{eq:state}, and
\eqref{eq:prehi}, we obtain
\[
 \sup_{0\leq t\leq T}\int_\Omega|q|^2\dd x
 +\int_0^T\!\int_\Omega|q|^2\dd x\dd t\leq C.
\]
Since $v_x=\frac{vq}{\mu}$ and both $v$ and $\mu$ are bounded above and away from zero, this proves
\eqref{eq:vx}.

In case $\eqref{eq:viscosity}_2$,
using $u_x=v_t$ in \eqref{eq:mom} and dividing by $\mu$, we obtain,
with $q=(\log v)_x$,
\begin{align*}
 q_t-\left(\frac u\mu\right)_t
 =-\frac{\theta q}{\mu v}+\frac{\theta_x}{\mu v}
   +\frac{b\cdot b_x}{\mu}
   +\frac{u\mu_t}{\mu^2}-\frac{\mu_xu_x}{\mu v}.
\end{align*}
Multiply this identity by $q$ and integrate over $\Omega$.
Since $q_t=\left(\frac{u_x}{v}\right)_x$ and $u=0$ at the half-line endpoint,
\[
 -\int_\Omega\frac u\mu q_t\dd x
 =\int_\Omega\frac{u_x^2}{\mu v}\dd x
  -\int_\Omega\frac{u\mu_xu_x}{\mu^2v}\dd x.
\]
Consequently,
\begin{align}
 &\frac{\dd}{\dd t}\int_\Omega
       \left\{\frac12q^2-\frac u\mu q\right\}\dd x
   +\int_\Omega\frac{\theta q^2}{\mu v}\dd x\notag\\
 &=\int_\Omega\left\{\frac{u_x^2}{\mu v}
          +\frac{\theta_xq}{\mu v}+\frac{q\,b\cdot b_x}{\mu}\right\}\dd x
      +\mathcal R_q(t)\notag\\
 &\leq\frac12\int_\Omega\frac{\theta q^2}{\mu v}\dd x
      +C\int_\Omega\left\{u_x^2+\frac{\theta_x^2}{v\theta}
                           +|b|^2|b_x|^2\right\}\dd x+|\mathcal R_q(t)|,
 \label{eq:lj09-T}
\end{align}
where the last inequality follows from \eqref{eq:state},
\eqref{eq:mu-basic}, and Young's inequality, and
\begin{equation}
 \mathcal R_q(t)=\int_\Omega\left\{
 \frac{u\mu_tq}{\mu^2}-\frac{\mu_xu_xq}{\mu v}
                      -\frac{u\mu_xu_x}{\mu^2v}\right\}\dd x.
 \label{eq:q-rem-T}
\end{equation}
Since $\theta\geq c$ is already known,
$|\mu_x|\leq C\gamma|\theta_x|$ and
$|\mu_t|\leq C\gamma|\theta_t|$. Moreover,
\eqref{eq:GN} and \eqref{eq:boot-norm} give
$\int_0^T\norm{u_x}_{L^\infty}^2\dd t\leq CM$.
Thus, using the temporary bounds only in the terms carrying $\gamma$,
\begin{align}
 \int_0^T|\mathcal R_q(t)|\dd t
 \leq C\gamma\biggl\{&
 \sup_{t\leq T}\norm u_{L^\infty}
 \left(\int_0^T\norm{\theta_t}_{L^2}^2\dd t\right)^{\frac{1}{2}}
 \left(\int_0^T\norm q_{L^2}^2\dd t\right)^{\frac{1}{2}}\notag\\
 &+\sup_{t\leq T}\norm{\theta_x}_{L^2}
 \left(\int_0^T\norm{u_x}_{L^\infty}^2\dd t\right)^{\frac{1}{2}}
 \left(\int_0^T\norm q_{L^2}^2\dd t\right)^{\frac{1}{2}}\notag\\
 &+\sup_{t\leq T}\norm u_{L^\infty}
 \left(\int_0^T\norm{\theta_x}_{L^2}^2\dd t\right)^{\frac{1}{2}}
 \left(\int_0^T\norm{u_x}_{L^2}^2\dd t\right)^{\frac{1}{2}}\biggr\}\notag\\
 &\leq C\gamma(1+M)^2\leq C\varepsilon_{\gamma}.
 \label{eq:q-rem-bd-T}
\end{align}

Finally, by \eqref{eq:ent-id},
\[
 \int_\Omega\left\{\frac12|(\log v)_x|^2
 -\frac u\mu(\log v)_x\right\}\dd x
 \geq\frac14\int_\Omega|(\log v)_x|^2\dd x
 -C\int_\Omega u^2\dd x
 \geq\frac14\int_\Omega|(\log v)_x|^2\dd x-C.
\]
Integrating \eqref{eq:lj09-T} over $(0,t)$ and using this lower bound,
\eqref{eq:init}, \eqref{eq:state},
\eqref{eq:prehi}, and \eqref{eq:q-rem-bd-T}, we obtain
\[
 \sup_{0\leq t\leq T}\int_\Omega|(\log v)_x|^2\dd x
 +\int_0^T\!\int_\Omega|(\log v)_x|^2\dd x\dd t\leq C.
\]
Since $v_x=v(\log v)_x$ and $c\leq v\leq C$, this proves
\eqref{eq:vx}.
\end{proof}

We next estimate the first and second derivatives of $u$.

\begin{proposition}\label{p:long}
For every fixed $\beta\geq0$,
\begin{align}
 &\sup_{0\leq t\leq T}\int_\Omega u_x^2\dd x+\int_0^T\!\int_\Omega
 \{u_{xx}^2+u_t^2+\theta_x^2\}\dd x\dd t\leq C.
                                                               \label{eq:long}
\end{align}
Moreover,
\begin{align}
 \int_0^T\biggl[&
 \norm{u_x(t)}_{L^\infty}^2+\norm{w_x(t)}_{L^\infty}^2
 +\norm{b_x(t)}_{L^\infty}^2\notag\\
 &+\{\norm{u_x(t)}_{L^\infty}^2+\norm{w_x(t)}_{L^\infty}^2
 +\norm{b_x(t)}_{L^\infty}^2\}^2\biggr]\dd t\leq C.
                                                               \label{eq:gmax}
\end{align}
\end{proposition}

\begin{proof}
Expanding \eqref{eq:mom}, we obtain
\begin{equation}
 u_t=\frac{\mu u_{xx}}{v}+\frac{\mu_xu_x}{v}-\frac{\mu u_xv_x}{v^2}-\frac{\theta_x}{v}
 +\frac{\theta v_x}{v^2}-b\cdot b_x.                \label{eq:momexp}
\end{equation}
Multiplying \eqref{eq:momexp} by $-u_{xx}$, integrating over
$\Omega$, and using \eqref{eq:bc-C}--\eqref{eq:bc-D}, we obtain
\begin{align*}
 \frac12\frac{\dd}{\dd t}\int_\Omega u_x^2\dd x
 +\int_\Omega\frac{\mu u_{xx}^2}{v}\dd x=\int_\Omega\left\{
 \frac{\mu u_xv_xu_{xx}}{v^2}+\frac{\theta_xu_{xx}}{v}
 -\frac{\theta v_xu_{xx}}{v^2}
 +(b\cdot b_x)u_{xx}\right\}\dd x+\mathcal R_u(t).
\end{align*}
Here
\[
 \mathcal R_u(t)=-\int_\Omega\frac{\mu_xu_xu_{xx}}v\dd x.
\]
In case $\eqref{eq:viscosity}_1$, $\mu_x=\mu'(v)v_x$, so
\[
 |\mathcal R_u(t)|
 \leq C\norm{v_x}_{L^2}\norm{u_x}_{L^\infty}\norm{u_{xx}}_{L^2}
 \leq\varepsilon\norm{u_{xx}}_{L^2}^2
       +C_\varepsilon\norm{u_x}_{L^2}^2.
\]
In case $\eqref{eq:viscosity}_2$, by \eqref{eq:state}, \eqref{eq:mu-basic}, and \eqref{eq:boot-norm},
\begin{align}
 \int_0^T|\mathcal R_u(t)|\dd t
 &\leq C\gamma\sup_{t\leq T}\norm{\theta_x}_{L^2}
        \left(\int_0^T\norm{u_x}_{L^\infty}^2\dd t\right)^{\frac{1}{2}}
        \left(\int_0^T\norm{u_{xx}}_{L^2}^2\dd t\right)^{\frac{1}{2}}\notag\\
 &\leq C\gamma M^{\frac{3}{2}}\leq C\varepsilon_{\gamma}.
 \label{eq:u-rem-bd}
\end{align}
For every $\varepsilon>0$, \eqref{eq:state},
\eqref{eq:GN}, and Young's inequality give
\begin{align*}
 \left|\int_\Omega\frac{\mu u_xv_xu_{xx}}{v^2}\dd x\right|
 &\leq C\norm{v_x}_{L^2}\norm{u_x}_{L^\infty}\norm{u_{xx}}_{L^2}\\
 &\leq C \norm{u_x}_{L^2}^{\frac12}
 \norm{u_{xx}}_{L^2}^{\frac32}\\
 &\leq\varepsilon\norm{u_{xx}}_{L^2}^2
 +C_\varepsilon \norm{u_x}_{L^2}^2,\\
 \left|\int_\Omega\frac{\theta_xu_{xx}}{v}\dd x\right|
 &\leq\varepsilon\norm{u_{xx}}_{L^2}^2
 +C_\varepsilon\int_\Omega\theta_x^2\dd x,\\
 \left|\int_\Omega\frac{\theta v_xu_{xx}}{v^2}\dd x\right|
 &\leq\varepsilon\norm{u_{xx}}_{L^2}^2
 +C_\varepsilon\int_\Omega\theta^2v_x^2\dd x,\\
 \left|\int_\Omega(b\cdot b_x)u_{xx}\dd x\right|
 &\leq \norm{u_{xx}}_{L^2} \norm{b}_{L^\infty} \norm{b_{x}}_{L^2}
  \leq\varepsilon\norm{u_{xx}}_{L^2}^2
 +C_\varepsilon\int_\Omega |b_x|^2\dd x.
\end{align*}
Choosing $\varepsilon>0$ sufficiently small and substituting these
estimates into the identity above, we obtain
\begin{align}
 \frac{\dd}{\dd t}\int_\Omega u_x^2\dd x
 +c\int_\Omega u_{xx}^2\dd x
 &\leq C \norm{u_x(t)}_{L^2}^2
 +C\int_\Omega\theta_x^2\dd x+C\norm{v_x(t)}_{L^2}^2\notag\\
 &\quad+C\norm{(\theta(t)-2)_+}_{L^\infty}^2
 +C\int_\Omega |b_x|^2\dd x+2|\mathcal R_u(t)|.
                                                        \label{eq:ux-est}
\end{align}
Here we used $\theta\leq2+(\theta-2)_+$ and
\[
 \int_\Omega\theta^2v_x^2\dd x \leq C\norm{v_x(t)}_{L^2}^2
 +C\norm{(\theta(t)-2)_+}_{L^\infty}^2\norm{v_x(t)}_{L^2}^2 \leq C\norm{v_x(t)}_{L^2}^2
 +C\norm{(\theta(t)-2)_+}_{L^\infty}^2.
\]

We claim that, for every fixed $\beta\geq0$,
\begin{align}
 \sup_{0\leq t\leq T}\int_\Omega u_x^2\dd x
 +\int_0^T\!\int_\Omega
 \{u_{xx}^2+\theta_x^2\}\dd x\dd t\leq C. \label{eq:lj10}
\end{align}

If $\beta=0$, \eqref{eq:hot-0} already gives
$\int_0^T\!\int_\Omega\theta_x^2\dd x\dd t\leq C$.
If $\beta\geq2$, the bounds $\theta\geq c$ and $v\leq C$ give
\[
 \int_0^T\!\int_\Omega\theta_x^2\dd x\dd t
 \leq C\int_0^T\!\int_\Omega
 \frac{\theta^{\beta-2}\theta_x^2}{v}\dd x\dd t\leq C.
\]
Integrating \eqref{eq:ux-est} over $(0,t)$ and using
\eqref{eq:prehi}, the preceding bounds for $\mathcal R_u$, and
\cref{l:vx,eq:excurs,eq:b-inf}, we absorb the multiple of
$\int_0^t\norm{u_{xx}}_{L^2}^2\dd s$ in case $\eqref{eq:viscosity}_1$
and obtain
\begin{align*}
 \sup_{0\leq t\leq T}\norm{u_x(t)}_{L^2}^2
 +\int_0^T\norm{u_{xx}(t)}_{L^2}^2\dd t\leq C,
 \qquad \beta=0\quad\text{or}\quad\beta\geq2.
\end{align*}

It remains to consider $0<\beta<2$.
Multiplying \eqref{eq:temp} by $(\theta-2)_+$ and
integrating over $\Omega$, we obtain
\begin{align}
 &\frac12\frac{\dd}{\dd t}\int_\Omega(\theta-2)_+^2\dd x
 +\int_{\{\theta>2\}}\frac{\theta^\beta\theta_x^2}{v}\dd x=-\int_\Omega\frac{\theta u_x}{v}(\theta-2)_+\dd x
 +\int_\Omega\frac{(\theta-2)_+
 (\mu u_x^2+|w_x|^2+|b_x|^2)}{v}\dd x.
                                                               \label{eq:therm}
\end{align}
By \eqref{eq:hot} and \eqref{eq:state},
\[
 \int_\Omega(\theta-2)_+\dd x\leq C,\qquad
 \int_\Omega\theta(\theta-2)_+\dd x
 \leq C\norm{(\theta(t)-2)_+}_{L^\infty}.
\]
Hence, for every $\varepsilon>0$, Young's inequality and
\eqref{eq:GN} give
\begin{align*}
 &\left|\int_\Omega\frac{\theta u_x}{v}(\theta-2)_+\dd x\right|
 +\int_\Omega\frac{(\theta-2)_+
 (\mu u_x^2+|w_x|^2+|b_x|^2)}{v}\dd x\\
 &\quad\leq C\{\norm{u_x(t)}_{L^\infty}^2+\norm{w_x(t)}_{L^\infty}^2
 +\norm{b_x(t)}_{L^\infty}^2\}
 +C\norm{u_x(t)}_{L^\infty}\norm{(\theta(t)-2)_+}_{L^\infty}\\
 &\quad\leq\varepsilon\norm{u_{xx}(t)}_{L^2}^2
 +C_\varepsilon\norm{u_x(t)}_{L^2}^2
 +C\{\norm{w_x(t)}_{L^\infty}^2+\norm{b_x(t)}_{L^\infty}^2\}
 +C\norm{(\theta(t)-2)_+}_{L^\infty}^2.
\end{align*}

Multiply \eqref{eq:therm} by $2\Lambda$ with
$\Lambda$ sufficiently large and add it to \eqref{eq:ux-est}.
Choosing $\varepsilon$ sufficiently small and using the positive
lower bound for $\frac{\theta^\beta}{v}$ on $\{\theta>2\}$,
we obtain
\begin{align}
 &\frac{\dd}{\dd t}\left\{
 \norm{u_x(t)}_{L^2}^2+\Lambda\int_\Omega(\theta-2)_+^2\dd x\right\}
 +c\int_\Omega u_{xx}^2\dd x
 +c\int_{\{\theta>2\}}\theta_x^2\dd x\notag\\
 &\quad\leq  C\biggl\{
 \norm{u_x(t)}_{L^2}^2
 +\int_{\{\theta\leq2\}}\theta_x^2\dd x
 + \norm{(\theta(t)-2)_+}_{L^\infty}^2 +\norm{v_x(t)}_{L^2}^2\notag\\
 &\qquad+\norm{b_x  }_{L^2}^2
 +\norm{w_x }_{L^\infty}^2+\norm{b_x }_{L^\infty}^2\biggr\}+2|\mathcal R_u(t)|.
                                                               \label{eq:uxheat}
\end{align}
Integrating \eqref{eq:uxheat} over $(0,t)$ and
using \eqref{eq:prehi}, \eqref{eq:excurs},
\cref{l:vx}, \eqref{eq:tr-max}, and the bounds for $\mathcal R_u$, we absorb the multiple of
$\int_0^t\norm{u_{xx}}_{L^2}^2\dd s$ in case $\eqref{eq:viscosity}_1$
and obtain
\begin{align*}
 \sup_{0\leq t\leq T}\norm{u_x(t)}_{L^2}^2
 +\int_0^T\norm{u_{xx}(t)}_{L^2}^2\dd t+\int_0^T\!\int_{\{\theta>2\}}
 \theta_x^2\dd x\dd t\leq C.
\end{align*}
Together with \eqref{eq:state} and
\eqref{eq:prehi}, this proves \eqref{eq:lj10} for every
fixed $\beta\geq0$.

To estimate $u_t$, in case $\eqref{eq:viscosity}_1$ we have
\[
 \int_0^T\!\int_\Omega|\mu_xu_x|^2\dd x\dd t
 \leq C\sup_{t\leq T}\norm{v_x}_{L^2}^2
       \int_0^T\norm{u_x}_{L^\infty}^2\dd t\leq C,
\]
where the last integral is bounded by \eqref{eq:GN} and
\eqref{eq:lj10}. In case $\eqref{eq:viscosity}_2$,
\begin{equation}
 \int_0^T\!\int_\Omega|\mu_xu_x|^2\dd x\dd t
 \leq C\gamma^2\sup_{t\leq T}\norm{\theta_x}_{L^2}^2
                    \int_0^T\norm{u_x}_{L^\infty}^2\dd t
 \leq C\gamma^2M^2\leq C.
 \label{eq:mu-source}
\end{equation}
Then we use
\begin{align*}
 \int_0^T\!\int_\Omega u_t^2\dd x\dd t
 &\leq C\int_0^T\!\int_\Omega
 \{u_{xx}^2+u_x^2v_x^2+\theta_x^2+\theta^2v_x^2
 +|b|^2|b_x|^2+|\mu_xu_x|^2\}\dd x\dd t\leq C,
\end{align*}
where in the first inequality we used \eqref{eq:momexp} and
\eqref{eq:state}, while in the last inequality we used the preceding bounds for $\mu_xu_x$,
\eqref{eq:GN}, \eqref{eq:cap},
\eqref{eq:excurs}, \eqref{eq:b-inf},
\cref{l:vx}, and
\begin{align*}
 \int_0^T\norm{u_x(t)}_{L^\infty}^2\dd t
 &\leq C\left(\int_0^T\norm{u_x(t)}_{L^2}^2\dd t\right)^{\frac12}
 \left(\int_0^T\norm{u_{xx}(t)}_{L^2}^2\dd t\right)^{\frac12}\leq C,\\
 \int_0^T\norm{u_x(t)}_{L^\infty}^4\dd t
 &\leq C\sup_{0\leq t\leq T}\norm{u_x(t)}_{L^2}^2
 \int_0^T\norm{u_{xx}(t)}_{L^2}^2\dd t\leq C,\\
 \int_0^T\!\int_\Omega u_x^2v_x^2\dd x\dd t
 &\leq\sup_{0\leq t\leq T}\norm{v_x(t)}_{L^2}^2
 \int_0^T\norm{u_x(t)}_{L^\infty}^2\dd t\leq C,\\
 \int_0^T\!\int_\Omega\theta^2v_x^2\dd x\dd t
 &\leq C\int_0^T\{\norm{v_x(t)}_{L^2}^2
 +\norm{(\theta(t)-2)_+}_{L^\infty}^2\}\dd t\leq C.
\end{align*}
This proves \eqref{eq:long}.
The bounds for $\norm{u_x(t)}_{L^\infty}$, together with
\eqref{eq:tr-max}, also give \eqref{eq:gmax}.
\end{proof}

We now turn to the temperature estimates.

\begin{proposition}\label{p:thig}
For every fixed $\beta\geq0$,
\begin{align}
 &\sup_{(x,t)\in\Omega\times[0,T]}\theta(x,t)\leq C,
                                                               \label{eq:t-up}\\
 &\sup_{0\leq t\leq T}\int_\Omega\theta_x^2\dd x
 +\int_0^T\!\int_\Omega
 \{\theta_t^2+\theta_{xx}^2\}\dd x\dd t\leq C.
                                                               \label{eq:t-H2}
\end{align}
\end{proposition}

\begin{proof}
If $\beta=0$, \eqref{eq:hot-0} and
\eqref{eq:hot} give
\begin{align*}
 \sup_{0\leq t\leq T}\int_{\{\theta>2\}}(1+\theta^2)\dd x&\leq C,\quad\int_0^T\!\int_{\{\theta>4\}}\theta_x^2\dd x\dd t\leq C.
\end{align*}

Suppose now that $\beta>0$.
Multiplying \eqref{eq:temp} by $(\theta-2)_+^{\beta+1}$
(cf.\ \cite{LiXu2026}),
we use the chain rule and the uniform bounds for $v$ to obtain
\begin{align}
 &\frac1{\beta+2}\frac{\dd}{\dd t}
 \int_\Omega(\theta-2)_+^{\beta+2}\dd x
 +c(\beta+1)\int_{\{\theta>2\}}
 \theta^\beta(\theta-2)_+^\beta\theta_x^2\dd x\notag\\&\leq C\{\norm{u_x(t)}_{L^\infty}^2
 +\norm{w_x(t)}_{L^\infty}^2+\norm{b_x(t)}_{L^\infty}^2\}
 \norm{(\theta(t)-2)_+}_{L^\infty}^\beta\notag\\
 &\quad+C\norm{u_x(t)}_{L^\infty}
 \norm{(\theta(t)-2)_+}_{L^\infty}^{\beta+1}.
                                                        \label{eq:httest}
\end{align}
Here we used $0\leq(\theta-2)_+\leq\theta$ and
\eqref{eq:hot}.
Applying the fundamental theorem of calculus and the Cauchy--Schwarz
inequality, we obtain
\begin{align}
 \norm{(\theta(t)-2)_+}_{L^\infty}^{2\beta+3}
 &\leq C\left(\int_{\{\theta>2\}}
 \frac{\theta^\beta(\theta-2)_+^\beta\theta_x^2}{v}\dd x\right)
 \left(\int_{\{\theta>2\}}v(\theta-2)_+^{\beta+1}
 \theta^{-\beta}\dd x\right)
 \notag\\
 &\leq C\int_{\{\theta>2\}}
 \theta^\beta(\theta-2)_+^\beta\theta_x^2\dd x.
                                               \label{eq:htmin}
\end{align}
In the last inequality, we used
$(\theta-2)_+^{\beta+1}\theta^{-\beta}\leq(\theta-2)_+$,
the uniform bounds for $v$, and \eqref{eq:hot}.

By Young's inequality, for every $\varepsilon>0$,
\begin{align*}
 &\{\norm{u_x(t)}_{L^\infty}^2+\norm{w_x(t)}_{L^\infty}^2
 +\norm{b_x(t)}_{L^\infty}^2\}
 \norm{(\theta(t)-2)_+}_{L^\infty}^\beta
  \notag\\&\leq\varepsilon
 \norm{(\theta(t)-2)_+}_{L^\infty}^{2\beta+3}
 +C_\varepsilon\{\norm{u_x(t)}_{L^\infty}^2
 +\norm{w_x(t)}_{L^\infty}^2+\norm{b_x(t)}_{L^\infty}^2\}^{\frac{2\beta+3}{\beta+3}},\\
& \norm{u_x(t)}_{L^\infty}
 \norm{(\theta(t)-2)_+}_{L^\infty}^{\beta+1}
 \leq\varepsilon
 \norm{(\theta(t)-2)_+}_{L^\infty}^{2\beta+3}
 +C_\varepsilon\{\norm{u_x(t)}_{L^\infty}^2
 +\norm{(\theta(t)-2)_+}_{L^\infty}^2\}.
\end{align*}
For the second inequality, we distinguish between
$\norm{(\theta(t)-2)_+}_{L^\infty}\leq1$ and
$\norm{(\theta(t)-2)_+}_{L^\infty}>1$.
Since $1<\frac{2\beta+3}{\beta+3}<2$,
\eqref{eq:gmax} and \eqref{eq:excurs} give
\begin{align*}
 &\{\norm{u_x(t)}_{L^\infty}^2+\norm{w_x(t)}_{L^\infty}^2
 +\norm{b_x(t)}_{L^\infty}^2\}^{\frac{2\beta+3}{\beta+3}}
 \in L^1(0,T),\\
 &\norm{(\theta(t)-2)_+}_{L^\infty}^2\in L^1(0,T).
\end{align*}
Both $L^1$ norms are bounded independently of $T$.
Taking $\varepsilon$ sufficiently small in
\eqref{eq:httest} and using
\eqref{eq:htmin} and the initial data, we obtain
\begin{align}
 \sup_{0\leq t\leq T}\int_\Omega(\theta-2)_+^{\beta+2}\dd x+\int_0^T\!\int_{\{\theta>2\}}
 \theta^\beta(\theta-2)_+^\beta\theta_x^2\dd x\dd t\leq C.
                                                               \label{eq:htmom}
\end{align}
On $\{\theta>4\}$, we have
$\frac{\theta}{2}\leq(\theta-2)_+\leq\theta$.
Together with \eqref{eq:hot} and
\eqref{eq:htmom}, this yields
\begin{align}
 &\sup_{0\leq t\leq T}\int_{\{\theta>2\}}
 \{\theta^\beta+\theta^{\beta+2}\}\dd x\leq C,
                                                               \label{eq:hot-Lp}\\
 &\int_0^T\!\int_{\{\theta>4\}}
 \theta^{2\beta}\theta_x^2\dd x\dd t\leq C.      \label{eq:htgrad}
\end{align}
In view of \eqref{eq:hot-0} and \eqref{eq:hot}
for $\beta=0$, \eqref{eq:hot-Lp} and
\eqref{eq:htgrad} hold for every fixed $\beta\geq0$.

Set
\[
 Y_\beta(t)=\int_\Omega
 \frac{\theta^{2\beta}\theta_x^2}{v}\dd x.
\]
Multiply \eqref{eq:temp} by $\theta^\beta\theta_t$
(cf.\ \cite{CaoPengSun2021,TongWangZhang2026} and, for $\beta=0$,
\cite{LuShiXiong2021}).
Since
\[
 (\theta^\beta\theta_t)_x=(\theta^\beta\theta_x)_t,
 \qquad v_t=u_x,
\]
we obtain
\begin{align}
 \frac12Y_\beta'(t)+\int_\Omega\theta^\beta\theta_t^2\dd x=-\frac12\int_\Omega
 \frac{\theta^{2\beta}\theta_x^2u_x}{v^2}\dd x
 -\int_\Omega\frac{\theta^{\beta+1}u_x\theta_t}{v}\dd x+\int_\Omega\frac{\theta^\beta
 (\mu u_x^2+|w_x|^2+|b_x|^2)\theta_t}{v}\dd x.
                                                               \label{eq:flux-t}
\end{align}
Young's inequality, the uniform positive lower bounds for $v$ and
$\theta$, and
$\norm{u_x(t)}_{L^\infty}\leq1+\norm{u_x(t)}_{L^\infty}^2$ give
\begin{align}
 Y_\beta'(t)&+c\int_\Omega\theta^\beta\theta_t^2\dd x
 \leq C\norm{u_x(t)}_{L^\infty}^2Y_\beta(t)+CY_\beta(t)+C\int_\Omega\biggl\{
 \theta^{\beta+2}u_x^2
 +\theta^\beta(\mu u_x^2+|w_x|^2+|b_x|^2)^2\biggr\}\dd x.
                                                               \label{eq:Y}
\end{align}

By \eqref{eq:hot-Lp}, for every fixed $\beta\geq0$,
\[
 \int_\Omega\theta^{\beta+2}u_x^2\dd x
 \leq C\norm{u_x(t)}_{L^2}^2+C\norm{u_x(t)}_{L^\infty}^2.
\]
Together with \eqref{eq:prehi} and
\eqref{eq:gmax}, this gives
\begin{equation}
 \int_0^T\!\int_\Omega
 \theta^{\beta+2}u_x^2\dd x\dd t\leq C.
                                                               \label{eq:t-ux}
\end{equation}

Using $\theta^\beta\leq C$ on $\{\theta\leq2\}$,
\eqref{eq:hot-Lp} on $\{\theta>2\}$, and
\cref{eq:trflux,eq:long}, we obtain
\begin{align}
 &\int_\Omega\theta^\beta
 (\mu u_x^2+|w_x|^2+|b_x|^2)^2\,\dd x\notag\\
 &\quad\leq C\{\norm{u_x(t)}_{L^\infty}^2+\norm{w_x(t)}_{L^\infty}^2
 +\norm{b_x(t)}_{L^\infty}^2\}
 \int_\Omega(\mu u_x^2+|w_x|^2+|b_x|^2)\,\dd x\notag\\
 &\qquad+C\{\norm{u_x(t)}_{L^\infty}^2+\norm{w_x(t)}_{L^\infty}^2
 +\norm{b_x(t)}_{L^\infty}^2\}^2
 \int_{\{\theta>2\}}\theta^\beta\,\dd x\notag\\
 &\quad\leq C\{\norm{u_x(t)}_{L^\infty}^2+\norm{w_x(t)}_{L^\infty}^2
 +\norm{b_x(t)}_{L^\infty}^2\}+C\{\norm{u_x(t)}_{L^\infty}^2+\norm{w_x(t)}_{L^\infty}^2
 +\norm{b_x(t)}_{L^\infty}^2\}^2.
                                                        \label{eq:mechwt}
\end{align}
By \eqref{eq:gmax}, the right-hand side is integrable
in time.

We also need to show that $Y_\beta$ is integrable in time.
On $\{\theta\leq4\}$,
$\theta^{2\beta}\leq C_\beta\theta^{-1}$, while on
$\{\theta>4\}$ we use \eqref{eq:htgrad}.
The uniform upper and lower bounds for $v$ and
\eqref{eq:prehi} therefore give
\begin{align*}
 \int_0^T Y_\beta(t)\dd t
 &\leq C_\beta\int_0^T\!\int_{\{\theta\leq4\}}
 \frac{\theta_x^2}{v\theta}\dd x\dd t
 +C\int_0^T\!\int_{\{\theta>4\}}\theta^{2\beta}\theta_x^2\dd x\dd t
 \leq C.
\end{align*}
Hence
\begin{equation}
 \int_0^T Y_\beta(t)\dd t\leq C.              \label{eq:Y-L1}
\end{equation}

Applying Gronwall's inequality to \eqref{eq:Y} and using
\cref{eq:init,eq:Y-L1,eq:t-ux,%
eq:mechwt,eq:gmax}, we obtain
\begin{equation}
 \sup_{0\leq t\leq T}Y_\beta(t)
 +\int_0^T\!\int_\Omega\theta^\beta\theta_t^2\dd x\dd t\leq C.
                                                               \label{eq:t-wt}
\end{equation}

We now prove the upper bound for the temperature.
For every fixed $\beta\geq0$, the fundamental theorem of calculus
and the Cauchy--Schwarz inequality give
\begin{align*}
 \norm{\bigl(\theta(t)^{\beta+1}-2^{\beta+1}\bigr)_+}_{L^\infty}
 &\leq(\beta+1)
 \left(\int_{\{\theta>2\}}
 \frac{\theta^{2\beta}\theta_x^2}{v}\dd x\right)^{\frac12}
 \left(\int_{\{\theta>2\}}v\dd x\right)^{\frac12}\leq C.
\end{align*}
Here we used \eqref{eq:t-wt}, the uniform upper bound
for $v$, and \eqref{eq:hot}.
This proves \eqref{eq:t-up}.
Combining \eqref{eq:t-wt} with the uniform positive
lower bound for $\theta$ and the uniform upper and lower bounds
for $v$, we obtain
\begin{equation}
 \sup_{0\leq t\leq T}\int_\Omega\theta_x^2\dd x\leq C,
 \qquad
 \int_0^T\!\int_\Omega\theta_t^2\dd x\dd t\leq C.
                                                               \label{eq:t-time}
\end{equation}

Finally, we expand \eqref{eq:temp} as
\begin{align}
 \theta_{xx}={}&v\theta^{-\beta}\theta_t
 +\theta^{1-\beta}u_x
 -\theta^{-\beta}(\mu u_x^2+|w_x|^2+|b_x|^2)-\beta\theta^{-1}\theta_x^2
 +\frac{v_x}{v}\theta_x.                            \label{eq:t-xx}
\end{align}
The uniform bounds for $v$ and $\theta$ and
\eqref{eq:t-xx} imply
\begin{align}
 \norm{\theta_{xx}(t)}_{L^2}^2
 &\leq C\int_\Omega\bigl\{\theta_t^2+u_x^2
 +(\mu u_x^2+|w_x|^2+|b_x|^2)^2+\theta_x^4+v_x^2\theta_x^2\bigr\}\dd x.
                                                        \label{eq:txxbd}
\end{align}
Moreover, the one-dimensional interpolation inequality and Young's
inequality give, for every $\varepsilon>0$,
\begin{align}
 \int_0^T\!\int_\Omega(\theta_x^4+v_x^2\theta_x^2)\dd x\dd t&\leq\varepsilon\int_0^T\norm{\theta_{xx}(t)}_{L^2}^2\dd t
 +C_\varepsilon\int_0^T\bigl\{\norm{\theta_x(t)}_{L^2}^6
 +\norm{v_x(t)}_{L^2}^4\norm{\theta_x(t)}_{L^2}^2\bigr\}\dd t\notag\\
 &\quad\leq\varepsilon\int_0^T\norm{\theta_{xx}(t)}_{L^2}^2\dd t+C_\varepsilon,
                                                        \label{eq:txxrem}
\end{align}
where the last inequality follows from \eqref{eq:long},
\eqref{eq:t-time}, and \cref{l:vx}. Combining
\cref{eq:txxbd,eq:txxrem,eq:long,%
eq:t-time,eq:mechwt,eq:gmax},
and choosing $\varepsilon$ sufficiently small, we obtain \eqref{eq:t-H2}.
\end{proof}

\section{Large-time behavior and proof of the main theorem}\label{sec:full}
\begin{proposition}\label{p:full}
For every fixed $\beta\geq0$,
\begin{align}
 \sup_{0\leq t\leq T}\bigl\{
 \norm{v(t)-1}_{H^1(\Omega)}^2+\norm{u(t)}_{H^1(\Omega)}^2
 +\norm{\theta(t)-1}_{H^1(\Omega)}^2+\norm{w(t)}_{H^1(\Omega)}^2+\norm{b(t)}_{H^1(\Omega)}^2\bigr\}\leq C,
                                                               \label{eq:fullH1}
\end{align}
and
\begin{align}
 \int_0^T\!\int_\Omega\bigl\{&
 v_x^2+\mu u_x^2+|w_x|^2+|b_x|^2+\theta_x^2\notag\\
 &+u_{xx}^2+|w_{xx}|^2+|b_{xx}|^2+\theta_{xx}^2\notag\\
 &+u_t^2+|w_t|^2+|b_t|^2+\theta_t^2\bigr\}
 \dd x\dd t\leq C.                                    \label{eq:fullds}
\end{align}
\end{proposition}

\begin{proof}
Using
\cref{eq:state,l:vx,%
p:long,p:thig} and
\cref{eq:ent-id,eq:trflux}, we obtain
\eqref{eq:fullH1}.

It remains to estimate the second spatial derivatives of $w$ and $b$.
Since
\begin{align*}
 w_{xx}=v\left(\frac{w_x}{v}\right)_x
 +v_x\frac{w_x}{v},\quad
 b_{xx}=v\left(\frac{b_x}{v}\right)_x
 +v_x\frac{b_x}{v},
\end{align*}
\cref{l:vx} and
\cref{eq:state,eq:trflux,eq:tr-max}
give
\begin{align*}
 \int_0^T\!\int_\Omega(|w_{xx}|^2+|b_{xx}|^2)\dd x\dd t&\leq C\int_0^T\!\int_\Omega
 \left\{\left|\left(\frac{w_x}{v}\right)_x\right|^2
 +\left|\left(\frac{b_x}{v}\right)_x\right|^2\right\}\dd x\dd t\\&+C\sup_{0\leq t\leq T}\norm{v_x(t)}_{L^2}^2
 \int_0^T\{\norm{w_x(t)}_{L^\infty}^2
 +\norm{b_x(t)}_{L^\infty}^2\}\dd t\leq C.
\end{align*}
Combining this estimate with
\cref{l:vx,p:cap,p:long,%
p:thig} and \eqref{eq:trflux} gives
\eqref{eq:fullds}.
\end{proof}
In case $\eqref{eq:viscosity}_1$, the bounds in Proposition~\ref{p:full} and
\eqref{eq:mainvt}, as established in Propositions~\ref{p:vlow},
\ref{p:tlow}, \ref{p:v-up}, and~\ref{p:thig}, are independent
of $T<T_*$. If $T_*<\infty$, then $v_t=u_x\in L^2(0,T_*;H^1)$
gives a limit for $v-1$ in $H^1$ at $T_*$. For the other variables,
the $L^2$ bounds for their time derivatives and second spatial
derivatives give continuity into $H^1$ up to $T_*$, with the
prescribed Dirichlet traces. The uniform positive lower bounds
pass to these terminal data. Lemma~\ref{lem:01a} therefore extends
the solution beyond $T_*$, a contradiction. Hence $T_*=\infty$.
Letting $T\to\infty$ in the preceding estimates gives their
global-in-time versions, which we use below.

In case $\eqref{eq:viscosity}_2$, we remove the temporary assumptions \eqref{eq:boot}.
The preceding estimates give constants $c_*,C_*>0$, independent of
$M,T,\gamma$, such that
\[
 c_*\leq v,\theta\leq C_*,\qquad \mathcal N_T\leq C_*.
\]
First fix a common $\varepsilon_*>0$ for the absorptions above.
Then choose $M_*>2$ so large that
\[
 c_*\geq\frac4{M_*},\quad C_*\leq \frac{M_*}{4},\quad
 \mathcal N_0\leq \frac{M_*}{4},
\]
and the initial values of $v$ and $\theta$ lie in $[\frac4{M_*},\frac{M_*}{4}]$.
Finally, set
\begin{equation}
 \gamma_0=\min\left\{1,\frac{\log2}{\log M_*},
                       \frac{\varepsilon_*}{(1+M_*)^8}\right\}>0.
 \label{eq:gamma-choice}
\end{equation}
For $0\leq\gamma\leq\gamma_0$, Lemma~\ref{lem:01a} and continuity
start \eqref{eq:boot} with strict inequalities and $M=M_*$.
On every interval where these hold, the preceding estimates improve
them to $\frac4{M_*}\leq v,\theta\leq \frac{M_*}{4}$ and
$\mathcal N_T\leq \frac{M_*}{4}$. They therefore cannot fail at a finite first
exit time. The continuation criterion in Lemma~\ref{lem:01a} excludes
a finite maximal existence time. The solution is consequently global,
and letting $T\to\infty$ gives \eqref{eq:fullH1} and
\eqref{eq:fullds} with the supremum and time integrals taken on
$[0,\infty)$, uniformly for $\gamma\in[0,\gamma_0]$.

\begin{proposition}\label{p:dec}
As $t\to\infty$,
\begin{equation}
 \norm{v_x(t)}_{L^2}+\norm{u_x(t)}_{L^2}+\norm{\theta_x(t)}_{L^2}
 +\norm{w_x(t)}_{L^2}+\norm{b_x(t)}_{L^2}\longrightarrow0.    \label{eq:grdec}
\end{equation}
Moreover, for every $2<p\leq\infty$,
\begin{align}
 \norm{v(t)-1}_{L^p(\Omega)}+\norm{u(t)}_{L^p(\Omega)}
 +\norm{\theta(t)-1}_{L^p(\Omega)}+\norm{w(t)}_{L^p(\Omega)}+\norm{b(t)}_{L^p(\Omega)}
 \longrightarrow 0.                                  \label{eq:Lp-dec}
\end{align}
 
\end{proposition}

\begin{proof}
By \eqref{eq:fullds}, the squares of all five $L^2$ norms
in \eqref{eq:grdec} belong to $L^1(0,\infty)$.
We estimate their time derivatives as in \cite{LiLiang2016}.
For $v$, the identity $v_{xt}=u_{xx}$ gives
\[
 \frac{\dd}{\dd t}\int_\Omega v_x^2\dd x
 =2\int_\Omega v_xu_{xx}\dd x.
\]
The right-hand side belongs to $L^1(0,\infty)$ by the
Cauchy--Schwarz inequality and \eqref{eq:fullds}.

Let $f$ be any scalar component of $u$, $w$, $b$, or $\theta-1$.
Integration by parts and the corresponding Dirichlet or Neumann
boundary condition give, for almost every $t$,
\[
 \frac{\dd}{\dd t}\int_\Omega f_x^2\dd x
 =-2\int_\Omega f_{xx}f_t\dd x.
\]
Again, the right-hand side belongs to $L^1(0,\infty)$ by
\eqref{eq:fullds}.
Thus, the squares of all five derivative norms belong to
$L^1(0,\infty)\cap W^{1,1}(0,\infty)$.
Each therefore has a limit as $t\to\infty$, and its integrability
forces that limit to be zero.
This proves \eqref{eq:grdec}.

Let $g$ be any scalar component of $v-1,u,\theta-1,w,b$.
By \eqref{eq:fullH1}, $\norm{g(t)}_{L^2}$ is uniformly bounded,
while \eqref{eq:grdec} gives
$\norm{g_x(t)}_{L^2}\to0$.
The one-dimensional Gagliardo--Nirenberg inequality
\eqref{eq:GN} yields
\[
 \norm{g(t)}_{L^p(\Omega)}
 \leq C_p\norm{g(t)}_{L^2}^{\frac12+\frac1p}
 \norm{g_x(t)}_{L^2}^{\frac12-\frac1p}\longrightarrow0
 \qquad(2<p\leq\infty).
\]
Summing over all scalar components proves \eqref{eq:Lp-dec}.

\end{proof}

\begin{proof}[Proof of \Cref{thm:main}]
Global existence in both cases was proved after Proposition~\ref{p:full}, and
uniqueness follows from Lemma~\ref{lem:01a}. The uniform upper
and lower bounds for $v$ and $\theta$ follow from
\cref{eq:state,%
p:thig}.
The estimates \eqref{eq:mainH1} and \eqref{eq:mainds} follow from
\cref{p:full}, and the decay statement 
\eqref{eq:maingr}  follows from
\cref{p:dec}.
\end{proof}

\appendix
\section{Local existence and justification of the energy estimates}
\label{sec:local}
\begin{proof}[Proof of Lemma~\ref{lem:01a}]
We give the construction for the normalized coefficients used in the
proof. The general positive coefficients are handled in the same way.
Write $Z=(u,w,b,\theta-1)$. On a fixed positive state interval
$m\leq v,\theta\leq M_1$, the equations have the form
\begin{equation}
 Z_t=A(v,\theta)Z_{xx}+F(v,\theta,v_x,Z,Z_x),
 \qquad
 A=\operatorname{diag}\left(\frac\mu v,\frac1v,
                           \frac1{v^2},\frac{\theta^\beta}{v}\right),
 \label{eq:local-system}
\end{equation}
where the entries for $w,b$ are repeated componentwise and
\begin{align*}
 F_u&=\frac{\mu_xu_x}{v}
       -\frac{\mu v_xu_x}{v^2}
       +\frac{\theta v_x}{v^2}-\frac{\theta_x}{v}-b\cdot b_x,\\
 F_w&=-\frac{v_xw_x}{v^2}+b_x,\\
 F_b&=-\frac{v_xb_x}{v^3}+\frac{w_x}{v}-\frac{u_xb}{v},\\
 F_\theta&=\frac{\beta\theta^{\beta-1}\theta_x^2}{v}
       -\frac{v_x\theta^\beta\theta_x}{v^2}
       +\frac{\mu u_x^2+|w_x|^2+|b_x|^2-\theta u_x}{v}.
\end{align*}
The matrix $A$ is uniformly positive, and its coefficients and their
state derivatives are bounded for each fixed $\alpha,\beta$
in $\eqref{eq:viscosity}_1$, and uniformly for $\gamma\in[0,1]$
in $\eqref{eq:viscosity}_2$, with $\beta$ fixed.
Here $\mu_x=\mu'(v)v_x$ in the first case and
$\mu_x=\mu_\theta\theta_x$ in the second.
For $g\in H^1(\Omega)$ and $f\in H^2(\Omega)$, the interpolation
inequality gives
\begin{equation}
 \norm{g_xf_x}_{L^2}^2
 \leq\varepsilon\norm{f_{xx}}_{L^2}^2
   +C_\varepsilon(1+\norm{g_x}_{L^2}^4)\norm{f_x}_{L^2}^2.
 \label{eq:local-product}
\end{equation}
This applies to each product of first derivatives in $F$, including
$\theta_x^2$, and does not require $v_{xx}$.
Set
\[
 H(t)=1+\norm{v-1}_{H^1}^2+\norm Z_{H^1}^2.
\]
Testing \eqref{eq:local-system} by $Z$ and $-Z_{xx}$, and using
$v_t=u_x$, we obtain, as long as the state remains in the prescribed
interval,
\begin{equation}
 H'(t)+c\norm{Z_{xx}}_{L^2}^2\leq C(1+H(t)^3),
 \qquad
 \norm{Z_t}_{L^2}^2\leq C\norm{Z_{xx}}_{L^2}^2+C(1+H(t)^3).
 \label{eq:local-est}
\end{equation}
For smooth functions, the boundary term in the derivative test is
zero: the Dirichlet value is constant in time, and the Neumann
derivative is zero.

Take smooth approximations of the initial data on
bounded intervals exhausting $\Omega$, with the far-field state at
the artificial endpoints. Use the eigenfunctions of the Laplacian
with the prescribed Dirichlet or Neumann condition for each component
of $Z$, and define the approximate specific volume by
$v^n=v_0^n+\int_0^t u_x^n\dd s$. Projection of
\eqref{eq:local-system} gives a finite system of ordinary differential
equations. The derivative test is admissible because $-Z_{xx}^n$
belongs to the same finite-dimensional space. Estimate
\eqref{eq:local-est} is uniform in the number of basis functions and
in the length of the interval. It gives a common short time on which
$H$ and the time integrals of $\norm{Z_{xx}}_{L^2}^2$ and
$\norm{Z_t}_{L^2}^2$ remain bounded. Moreover,
\[
 \norm{v(t)-v_0}_{H^1}\leq
 t^{\frac{1}{2}}\left(\int_0^t\norm{u_x}_{H^1}^2\dd s\right)^{\frac{1}{2}},
 \qquad
 \norm{Z(t)-Z_0}_{L^\infty}\leq Ct^{\frac{1}{4}},
\]
where the second inequality follows by interpolating the $L^2$ time
modulus with the bounded $H^1$ norm. These estimates keep $v,\theta$
inside a fixed positive state interval on a common time $[0,T_0]$.

Compactness on bounded space intervals gives strong convergence of
$Z^n$ in $L^2(0,T_0;H^1_{\rm loc})$ and locally uniform convergence
of the states. The derivatives $v_x^n$ converge weakly in $L^2$;
thus their products with $Z_x^n$ converge in distributions, while
\eqref{eq:local-product} bounds these products in $L^2$.
First letting the Galerkin dimension tend to infinity and then
exhausting $\Omega$ yields a solution with all the global norms in
\eqref{eq:finite}. The equations give $Z_t\in L^2(0,T_0;L^2)$.
The interpolation identity for the Dirichlet or Neumann Laplacian
therefore gives $Z\in C([0,T_0];H^1)$; its Neumann form domain is
$H^1$, so no initial derivative trace is imposed.
Also $v_t=u_x\in L^2(0,T_0;H^1)$ implies $v-1\in C([0,T_0];H^1)$.

To prove uniqueness, subtract the equations for two solutions in the
same positive state interval and apply \eqref{eq:local-product}.
With
$E=\norm{v^{(1)}-v^{(2)}}_{H^1}^2+
\norm{Z^{(1)}-Z^{(2)}}_{H^1}^2$, the same tests give
\begin{equation}
 E'(t)+c\norm{Z_{xx}^{(1)}-Z_{xx}^{(2)}}_{L^2}^2
 \leq C\left(1+\sum_{i=1}^2\norm{Z_{xx}^{(i)}(t)}_{L^2}^2\right)E(t),
 \label{eq:local-difference}
\end{equation}
where $C$ may depend on the common $H^1$ and state bounds.
The coefficient multiplying $E$ is integrable in time.
For example, the difference of the principal coefficients is bounded
in $L^\infty$ by $CE^{\frac{1}{2}}$; its product with $Z_{xx}^{(2)}$ is
controlled by the right side of \eqref{eq:local-difference}.
The remaining products are treated by expanding one factor at a time
and using \eqref{eq:local-product} and the bounded $H^1$ norms.
Gronwall's inequality proves uniqueness and continuous dependence.
Restarting the construction proves the continuation criterion, since
all constants above depend only on the stated norms and positive
bounds, for fixed $\alpha,\beta$ in $\eqref{eq:viscosity}_1$
and uniformly for $\gamma\in[0,1]$ in $\eqref{eq:viscosity}_2$.
\end{proof}

The energy identities above are justified by approximation in the
class \eqref{eq:finite}. Smooth data with homogeneous Dirichlet values
or with constant values near a Neumann endpoint are dense in the
corresponding $H^1$ space; truncation, a shrinking endpoint collar,
and mollification preserve the positive state bounds. Estimate
\eqref{eq:local-difference} gives convergence of their solutions in
$C_tH^1$ and in $L^2_tH^2$ for the diffusive variables, and the
equations give convergence of their time derivatives in $L^2$.
For a component $f$ of $u,w,b,\theta-1$, time convolution preserving
the homogeneous boundary condition consequently gives
\[
 \frac{\dd}{\dd t}\int_\Omega\frac{|f_x|^2}{v}\dd x
 =-2\int_\Omega\left(\frac{f_x}{v}\right)_xf_t\dd x
   -\int_\Omega\frac{u_x|f_x|^2}{v^2}\dd x.
\]
Indeed, $\left(\frac{f_x}{v}\right)_x\in L^2_{t,x}$ follows from
\eqref{eq:local-product}, while $u_x\in L^2_tL^\infty_x$.
Thus the formula does not require a boundary trace of $f_t$.
The same argument gives the identity for $\norm{f_x}_{L^2}^2$ used
in Proposition~\ref{p:dec}. For the heat time-derivative test, we apply the same argument to
$\frac{\theta^{\beta+1}-1}{\beta+1}$, which belongs to the
same strong class on every finite positive state interval.
Lipschitz approximations of the positive part justify the temperature
truncations, and spatial cutoffs remove the boundary at infinity.
The half-line boundary fluxes vanish by $u=w=0$, $\theta_x=0$,
and either $b=0$ or $b_x=0$. The energy identities leading to \eqref{eq:lj09} and
\eqref{eq:lj09-T} follow by the same approximation. In case
$\eqref{eq:viscosity}_1$, $q_t=\left(\frac{\mu u_x}{v}\right)_x\in L^2_{t,x}$
and $u\in C_tH^1_x$ on every finite interval. In case
$\eqref{eq:viscosity}_2$, $q_t=\left(\frac{u_x}{v}\right)_x\in L^2_{t,x}$
and $\frac u\mu\in C_tH^1_x$. The stresses in
\eqref{eq:sigma} and \eqref{eq:sigma-T} satisfy
$\sigma+1\in L^2(0,T;H^1(\Omega))$ in their respective cases.
Their endpoint traces and the differentiations and time integrals
in Lemmas~\ref{l:Kaz} and~\ref{l:Kaz-T} are therefore well defined.

\end{document}